\documentclass{amsart}

\usepackage{amsmath,amsthm,amssymb}
\usepackage{enumitem}
\usepackage{graphicx}
\usepackage{float}
\usepackage[margin=1.5in]{geometry}
\numberwithin{equation}{section}

\theoremstyle{plain}
\newtheorem{thm}{Theorem}[section]

\newtheorem{lem}[thm]{Lemma}
\newtheorem{prop}[thm]{Proposition}

\theoremstyle{definition}
\newtheorem{defn}[thm]{Definition}

\theoremstyle{remark}
\newtheorem{rem}[thm]{Remark}

\usepackage{tikz}
\usetikzlibrary{arrows,shapes,trees}
\usepackage{graphicx}
\usepackage{hyperref}
\usepackage{amssymb}
\usepackage{amsfonts}
\usepackage{amsthm}
\usepackage{amsmath}
\usepackage{epstopdf}
\usepackage{mathrsfs}

\newcommand{\N}{\mathbb N}

\newcommand{\R}{\mathbb R}

\newcommand{\e}{\varepsilon}

\newcommand{\dist}{\mathrm{dist}}
\newcommand{\diam}{\mathrm{diam}}
\newcommand{\ls}{\lesssim}
\newcommand{\gs}{\gtrsim}

\newcommand{\la}{\lambda}

\newcommand{\p}{\partial}

\def\Xint#1{\mathchoice
{\XXint\displaystyle\textstyle{#1}}%
{\XXint\textstyle\scriptstyle{#1}}%
{\XXint\scriptstyle\scriptscriptstyle{#1}}%
{\XXint\scriptscriptstyle\scriptscriptstyle{#1}}%
\!\int}
\def\XXint#1#2#3{{\setbox0=\hbox{$#1{#2#3}{\int}$ }
\vcenter{\hbox{$#2#3$ }}\kern-.6\wd0}}

\def\dashint{\Xint-}

\title[Doubling and Poincar\'e inequalities for uniformized measures]{Doubling and Poincar\'e inequalities for uniformized measures on Gromov hyperbolic spaces}
\author{Clark Butler}

\begin{document}
\begin{abstract}
We generalize the recent results of Bj\"orn-Bj\"orn-Shanmugalingam \cite{BBS20} concerning how measures transform under the uniformization procedure of Bonk-Heinonen-Koskela for Gromov hyperbolic spaces \cite{BHK} by showing that these results also hold in the setting of uniformizing Gromov hyperbolic spaces by Busemann functions that we introduced in \cite{Bu20}. In particular uniformly local doubling and uniformly local Poincar\'e inequalities for the starting measure transform into global doubling and global Poincar\'e inequalities for the uniformized measure. We then show in the setting of uniformizations of universal covers of closed negatively curved Riemannian manifolds equipped with the Riemannian measure that one can obtain sharp ranges of exponents for the uniformized measure to be doubling and satisfy a $1$-Poincar\'e inequality. Lastly we introduce the procedure of uniform inversion for uniform metric spaces, and show that both the doubling property and the $p$-Poincar\'e inequality are preserved by uniform inversion for any $p \geq 1$. 
\end{abstract}

\maketitle

\section{Introduction}

Broadly speaking our work in this paper has three closely related objectives. The primary objective is to generalize the recent results of Bj\"orn-Bj\"orn-Shanmugalingam \cite{BBS20} concerning how measures transform under the uniformization procedure of Bonk-Heinonen-Koskela for Gromov hyperbolic spaces \cite{BHK}. We will consider how measures transform under the generalization of this uniformization procedure that we introduced in \cite{Bu20}. As in \cite{BBS20}, we will show that uniformizing these measures upgrades uniformly local doubling properties and uniformly local Poincar\'e inequalities to global doubling and global Poincar\'e inequalities for the uniformized space. Our results allow us to construct a number of interesting new unbounded metric measure spaces supporting Poincar\'e inequalities. A particularly important example is uniformizations of hyperbolic fillings of unbounded metric spaces, which play a key role in our followup work \cite{Bu23} concerning extension and trace theorems for Besov spaces on noncompact doubling metric measure spaces. 

Our second objective is to show that in the presence of a cocompact isometric discrete group action on the Gromov hyperbolic space we start with, it is often possible to apply the theorems of \cite{BBS20} to a much wider range of exponents than the ones considered there, leading in several cases to ranges that we can verify are sharp due to well-known results on Patterson-Sullivan measures in these contexts. In Theorem \ref{Riem doubling} we tie this threshold to the \emph{volume growth entropy} of universal covers of closed negatively curved Riemannian manifolds. In Remark \ref{renormalize} we briefly explain how Patterson-Sullivan measures on the Gromov boundary arise as renormalized limits of the uniformized measures considered here and in \cite{BBS20}. 

The final topic that we consider is procedures for transforming metric measure spaces that preserve the doubling property and $p$-Poincar\'e inequalities for a given $p \geq 1$. Oftentimes in analysis on metric spaces it is preferable to work on either a bounded or an unbounded space depending on the nature of the question under consideration. Thus there has been a significant amount of interest in procedures for passing back and forth between bounded and unbounded spaces while retaining as much information as possible. In the abstract metric space setting transformations between bounded and unbounded spaces can be realized through \emph{inversions} \cite{BHX08}, which generalize the classical notion of M\"obius inversions in the complex plane. It was shown by Li-Shanmugalingam \cite{LS15} that measures can be transformed under these inversions in such a way that a number of desirable properties can be preserved. However they were not able to obtain unconditional invariance of Poincar\'e inequalities under inversions, and they in fact showed that Poincar\'e inequalities cannot always be preserved \cite[Example 3.3.13]{LS15}. In the final section of this paper we introduce an alternative inversion operation based on uniformization that is specialized to uniform metric spaces. With the additional assistance of some results from \cite{BBS20} we will show that this operation preserves Poincar\'e inequalities. A more in-depth discussion of this is given in Section \ref{sec:inversion}.

We now introduce the setting of \cite{Bu20} and \cite{BBS20} in order to state our main theorems. For precise definitions regarding general notions in Gromov hyperbolic spaces we refer to Section \ref{sec:hyperbolic}, while for a more detailed treatment of the uniformization procedure discussed in this introduction we refer to Section \ref{sec:uniformize}. We begin with a proper geodesic $\delta$-hyperbolic metric space $(X,d)$, meaning that geodesic triangles in $X$ are $\delta$-thin for a constant $\delta \geq 0$. The \emph{Gromov boundary} $\p X$ of $X$ is the collection of all geodesic rays $\gamma:[0,\infty) \rightarrow X$ up to the equivalence relation that two geodesic rays are equivalent if they are at bounded distance from one another. For a given geodesic ray $\gamma:[0,\infty) \rightarrow X$, the \emph{Busemann function} $b_{\gamma}: X \rightarrow \R$ associated to $\gamma$ is defined by the limit
\begin{equation}\label{first busemann definition}
b_{\gamma}(x) = \lim_{t \rightarrow \infty} d(\gamma(t),x)-t. 
\end{equation} 
We then define
\begin{equation}\label{extension busemann definition}
\mathcal{B}(X) = \{b_{\gamma}+s: \text{$\gamma$ a geodesic ray in $X$, $s \in \R$}\},
\end{equation}
and refer to any function $b \in \mathcal{B}(X)$ as a Busemann function on $X$. See \cite[(1.4-1.5)]{Bu20} for further details. The Busemann functions $b \in \mathcal{B}(X)$ are all $1$-Lipschitz functions on $X$. For a Busemann function $b$ of the form $b = b_{\gamma} + s$ for some $s \in \R$, we define the endpoint $\omega \in \p X$ of $\gamma$ to be the \emph{basepoint} of $b$ and say that $b$ is \emph{based at $\omega$}.  

For $z \in X$ we define $b_{z}(x) = d(x,z)$ to be the distance from $z$. We augment the set of Busemann functions with the set of translates of distance functions on $X$, 
\begin{equation}\label{distance definition}
\mathcal{D}(X) = \{b_{z}+s :z\in X, s\in \R\}.
\end{equation}
For $b \in \mathcal{D}(X)$ with $b = b_{z}+s$ for some $z \in X$ and $s \in \R$ we then refer to $z$ as the \emph{basepoint} of $b$, in analogy to the case of Busemann functions. We write $\hat{\mathcal{B}}(X) = \mathcal{D}(X) \cup \mathcal{B}(X)$. Then all functions $b \in \hat{\mathcal{B}}(X)$ are $1$-Lipschitz. 

For each $b \in \hat{\mathcal{B}}(X)$ and each $\e > 0$ we define a positive density $\rho_{\e,b}$ on $X$ by 
\[
\rho_{\e,b}(x) = e^{-\e b(x)}.
\]
For a curve $\gamma$ in $X$ we let
\[
\ell_{\e,b}(\gamma) = \int_{\gamma} \rho_{\e,b}\,ds,
\]
denote the line integral of $\rho_{\e,b}$ along $\gamma$. We let $(X_{\e,b},d_{\e,b})$ denote the metric space obtained by conformally deforming $X$ by the density $\rho_{\e,b}$, i.e., defining the new distance $d_{\e,b}$ for $x,y \in X$ by
\begin{equation}\label{define distance}
d_{\e,b}(x,y) = \inf \ell_{\e,b}(\gamma),
\end{equation}
with the infimum taken over all curves $\gamma$ joining $x$ to $y$. The metric sapce $X_{\e,b}$ is bounded if and only if $b \in \mathcal{D}(X)$ \cite[Proposition 4.4]{Bu20}. 

Our main results concern the corresponding effect of this conformal deformation on measures on $X$. We will require the following key definition. For the entirety of this paper a \emph{metric measure space} is a triple $(X,d,\mu)$ consisting of a metric space $(X,d)$ equipped with a Borel measure $\mu$. 

\begin{defn}\label{defn:local doubling}
Let $(X,d,\mu)$ be a metric measure space and let $B_{X}(x,r)$ denote the open ball of radius $r > 0$ centered at $x \in X$. The measure $\mu$ is \emph{doubling} if there is a constant $C_{\mu} \geq 1$ such that for every $x \in X$ and $r > 0$ we have
\begin{equation}\label{define doubling measure}
\mu(B_{X}(x,2r)) \leq C_{\mu}\mu(B_{X}(x,r)).
\end{equation}
If the inequality \eqref{define doubling measure} only holds for balls of radius at most $R_{0}$ then we will say that $\mu$ is \emph{doubling on balls of radius at most $R_{0}$}.  We will alternatively say that $\mu$ is \emph{uniformly locally doubling} if there is an $R_{0} > 0$ such that $\mu$ is doubling on balls of radius at most $R_{0}$.
\end{defn}

We let $\mu$ be a given Borel measure on $X$ that is doubling on balls of radius at most $R_{0}$ with constant $C_{\mu}$. For each $\beta > 0$ we define a measure $\mu_{\beta,b}$ on $X$ by
\begin{equation}\label{define beta}
d\mu_{\beta,b}(x) = \rho_{\beta,b}(x) d\mu(x) = e^{-\beta b(x)} d\mu(x) ,
\end{equation}
for $x \in X$. We will consider $\mu_{\beta,b}$ as a measure on $X_{\e,b}$ and extend it to the completion $\bar{X}_{\e,b}$ by setting $\mu_{\beta,b}(\p X_{\e,b}) = 0$, where $\p X_{\e,b} = \bar{X}_{\e,b} \backslash X_{\e,b}$ denotes the complement of $X_{\e,b}$ inside its completion. Our first theorem shows that there is a threshold $\beta_{0}$ depending only on $R_{0}$ and $C_{\mu}$ such that if $\beta \geq \beta_{0}$ then the uniformly locally doubling measure $\mu$ on $X$ transforms into a measure $\mu_{\beta,b}$ on $\bar{X}_{\e,b}$ that is doubling at all scales. 

Our theorem requires two additional hypotheses on $X$ and the density $\rho_{\e,b}$, which we briefly summarize here. We recall that we are assuming that $X$ is a proper geodesic $\delta$-hyperbolic space and that $b \in \hat{\mathcal{B}}(X)$. The first is that $X$ is \emph{$K$-roughly starlike} from the basepoint $\omega \in X \cup \p X$ of the chosen function  $b \in \hat{\mathcal{B}}(X)$ for a given constant $K \geq 0$. Roughly speaking this condition requires that each point $x \in X$ lies within distance $K$ of a geodesic ray or line starting from $\omega$. We defer a precise definition to Section \ref{sec:hyperbolic} as one must distinguish the cases $\omega \in X$ and $\omega \in \p X$. 

The second requirement is that $\rho_{\e,b}$ is a \emph{Gehring-Hayman density} for $X$ with constant $M \geq 1$ (abbreviated as \emph{GH-density}). This means that for each $x,y \in X$ and each \emph{geodesic} $\gamma$ joining $x$ to $y$ we have
\begin{equation}\label{GH inequality}
d_{\e,b}(x,y) \leq M\ell_{\e,b}(\gamma).
\end{equation}
In other words, $\rho_{\e,b}$ is a GH-density  if geodesics in $X$ minimize distance in $X_{\e,b}$ up to a universal multiplicative constant. This requirement is not as stringent as it first appears; by the work of Bonk-Heinonen-Koskela \cite[Theorem 5.1]{BHK} there is an $\e_{0} = \e_{0}(\delta) > 0$ depending only on $\delta$ such that $\rho_{\e,b}$ is a GH-density with constant $M = 20$ for any $b \in \hat{\mathcal{B}}(X)$ and any $0 < \e \leq \e_{0}$. For CAT$(-1)$ spaces one may use a threshold $\e_{0} = 1$ instead \cite[Theorem 1.10]{Bu20}.  The rough starlikeness hypothesis and the GH-density hypothesis together guarantee that the conformal deformation $X_{\e,b}$ has a number of nice properties by our previous results in \cite{Bu20} that are summarized in Section \ref{sec:uniformize} and are used heavily in the proofs of our theorems. 

We then have the following theorem; below ``the data" refers to the collection of parameters $\delta$, $K$, $\e$, $M$, $\beta$, $R_{0}$, and $C_{\mu}$.

\begin{thm}\label{doubling threshold}
There is $\beta_{0} = \beta_{0}(R_{0},C_{\mu}) > 0$ such that if $\beta \geq \beta_{0}$ then the measure $\mu_{\beta,b}$ on $\bar{X}_{\e,b}$ is doubling with constant $C_{\mu_{\beta}}$ depending only on the data. 
\end{thm}

This theorem generalizes the the main result of Bj\"orn-Bj\"orn-Shanmugalingam \cite[Theorem 1.1]{BBS20} to the setting in which the uniformization $X_{\e,b}$ is potentially unbounded, i.e., the case $b \in \mathcal{B}(X)$. The case $b \in \mathcal{D}(X)$ is more or less already contained in \cite[Theorem 1.1]{BBS20}, with the exception that they only consider the original parameter range $0 < \e \leq \e_{0}$ of Bonk-Heinonen-Koskela. An explicit value $\beta_{0} = \frac{17 \log C_{\mu}}{3R_{0}}$ is given in  \cite[Theorem 1.1]{BBS20}; a similar explicit estimate for $\beta_{0}$ can be extracted from our proof. In the process of proving Theorem \ref{doubling threshold} we formulate a useful criterion (Proposition \ref{global doubling}) for checking that $\mu_{\beta,b}$ is doubling on $\bar{X}_{\e,b}$. This criterion will be used to verify Theorem \ref{Riem doubling} below as well as some key claims in our followup work \cite{Bu23}.




The doubling property for $\mu_{\beta,b}$ on $\bar{X}_{\e,b}$ is the key property needed to transform uniformly local $p$-Poincar\'e inequalities on $X$ into global $p$-Poincar\'e inequalities on $X_{\e,b}$. The following theorem makes this claim precise. We refer to Section \ref{sec:poincare} for the precise definitions of uniformly local $p$-Poincar\'e inequalities and global $p$-Poincar\'e inequalities. We retain the same hypotheses regarding rough starlikeness and the GH-density property that we assumed in Theorem \ref{doubling threshold}. Below ``the data" refers to the parameters $\delta$, $K$, $\e$, $M$, $\beta$, the doubling constant $C_{\mu_{\beta,b}}$ for $\mu_{\beta,b}$ on $\bar{X}_{\e,b}$, the power $p$, the radius and the constants $C_{\mathrm{PI}}$ and $\la$ appearing in the uniformly local $p$-Poincar\'e inequality \eqref{uniformly local Poincare} as well as the local doubling radius and constant $C_{\mu}$ for $\mu$. 

\begin{thm}\label{global Poincare}
Suppose that the metric measure space $(X,d,\mu)$ is uniformly locally doubling and supports a uniformly local $p$-Poincar\'e inequality for some $p \geq 1$. Suppose further that for a given $\beta > 0$ we have that $\mu_{\beta,b}$ is doubling on $\bar{X}_{\e,b}$ with constant $C_{\mu_{\beta}}$. Then the metric measure spaces $(X_{\e,b},d_{\e,b},\mu_{\beta,b})$ and $(\bar{X}_{\e,b},d_{\e,b},\mu_{\beta,b})$ each support a $p$-Poincar\'e inequality with constant $C_{\mathrm{PI}}^{*}$ depending only on the data. 
\end{thm}

By combining Theorems \ref{doubling threshold} and \ref{global Poincare}, we see that if we assume $(X,d,\mu)$ is uniformly locally doubling and supports a uniformly local $p$-Poincar\'e inequality then for $\beta \geq \beta_{0}$ we always have that $(X_{\e,b},d_{\e,b},\mu_{\beta,b})$ is doubling and supports a $p$-Poincar\'e inequality, and the same is true with $\bar{X}_{\e,b}$ replacing $X_{\e,b}$. For $b \in \mathcal{D}(X)$ Theorem \ref{global Poincare} essentially follows directly from \cite[Theorem 1.1]{BBS20} and its proof. For Busemann functions $b \in \mathcal{B}(X)$ our uniformization construction in \cite{Bu20} is designed such that minimal modifications to the proofs in \cite{BBS20} are required. We emphasize that Theorem \ref{global Poincare} does not require us to restrict to the range $\beta \geq \beta_{0}$ considered in Theorem \ref{global doubling}; it only requires that $\mu_{\beta,b}$ is doubling on $\bar{X}_{\e,b}$. 

As indicated previously, when $X$ comes equipped with a cocompact discrete isometric group action it is possible to significantly improve Theorem \ref{doubling threshold} by obtaining a better, often sharp threshold $\beta_{0}$ for $\mu_{\beta,b}$ to be doubling. This is the content of Theorem \ref{crit doubling} in Section \ref{sec:doubling}. We highlight here an interesting corollary of this theorem that illustrates the power of this method. 

We consider a complete simply connected $n$-dimensional Riemannian manifold $X$ with sectional curvatures $\leq -1$ together with a cocompact discrete isometric action of a group $\Gamma$ on $X$. We let $\mu$ denote the Riemannian volume on $X$, which is $\Gamma$-invariant. The \emph{volume growth entropy} of $X$ is given by the limit for any $x \in X$, 
\begin{equation}\label{volume growth}
h_{X} = \lim_{R \rightarrow \infty} \frac{\log \mu(B_{X}(x,R))}{R}.
\end{equation}
For the existence of this limit see \cite{Man}. The quantity $h_{X}$ shows up in many places, for instance it is also equal to the topological entropy of the geodesic flow on the unit tangent bundle of the quotient of $X$ by $\Gamma$ \cite{Man}. The constants in Theorem \ref{Riem doubling} are uniform in the sense that they do not depend on the choice of function $b \in \hat{\mathcal{B}}(X)$.

\begin{thm}\label{Riem doubling}
For each $\beta > h_{X}$ the metric measure spaces $(X_{1,b},d_{1,b},\mu_{\beta,b})$ and $(\bar{X}_{1,b},d_{1,b},\mu_{\beta,b})$ for $b \in \hat{\mathcal{B}}(X)$ are doubling and support a $1$-Poincar\'e inequality with uniform constants. 
\end{thm} 

In Remark \ref{renormalize} we explain why the threshold $h_{X}$ is sharp. The constants in Theorem \ref{Riem doubling} are uniform in the sense that they do not depend on the choice of $b \in \hat{\mathcal{B}}(X)$, although they will depend on the choice of exponent $\beta > h_{X}$.

Our results regarding preservation of Poincar\'e inequalities for uniform metric spaces are proved in Section \ref{sec:inversion}. These results follow formally by combining Theorems \ref{doubling threshold} and \ref{global Poincare} above with results in \cite{Bu21} and \cite{BBS20}; since the initial setup is quite different from that of our theorems above we have isolated the discussion of those results to Section \ref{sec:inversion}. The rest of the paper is structured as follows: in Section \ref{sec:uniformize} we review some results from our previous work \cite{Bu20} regarding uniformizing Gromov hyperbolic spaces and extend some results from \cite{BBS20} to the setting of uniformizing by Busemann functions. In Section \ref{sec:doubling} we analyze the doubling properties of the uniformized measure \eqref{define beta} and prove Theorem \ref{doubling threshold}. Lastly in Section \ref{sec:poincare} we prove Theorems \ref{global Poincare} and \ref{Riem doubling}. 

\section{Uniformization}\label{sec:uniformize}

\subsection{Definitions}\label{subsec:defn} Let $X$ be a set and let $f$, $g$ be real-valued functions defined on $X$. For $c \geq 0$ we will write $f \doteq_{c} g$ if
\[
|f(x)-g(x)| \leq c,
\] 
for all $x \in X$. If the exact value of the constant $c$ is not important or implied by context we will often just write $f \doteq g$. The relation $f \doteq g$ will sometimes be referred to as a \emph{rough equality} between $f$ and $g$. Similarly for $C \geq 1$ and functions $f,g:X \rightarrow (0,\infty)$, we will write $f \asymp_{C} g$ if for all $x \in X$, 
\[
C^{-1}g(x) \leq f(x) \leq C g(x).
\]
We will write $f \asymp g$ if the value of $C$ is implied by context. We will write $f \ls_{C} g$ if $f(x) \leq Cg(x)$ for all $x \in X$ and $f \gs_{C} g$ if $f(x) \geq C^{-1}g(x)$ for $x \in X$. Thus $f \asymp_{C} g$ if and only if $f \ls_{C} g$ and $f \gs_{C} g$. As with the other notation, we will drop the constant $C$ and just write $f \ls g$ or $f \gs g$ if the value of $C$ is implied by context. We will generally stick to the convention of using $c \geq 0$ for additive constants and $C \geq 1$ for multiplicative constants. To indicate on what parameters -- such as $\delta$ -- the constants depend on we will write $c = c(\delta)$, etc. At the beginning of each section we will indicate on what parameters the implied constants of the inequalities $\ls$ and $\gs$, the comparisons $\asymp$, and the rough equalities $\doteq$ are allowed to depend. We will often reiterate these conditions for emphasis. 

For a metric space $(X,d)$ we will write $B_{X}(x,r) = \{y \in X:d(x,y) < r\}$ for the open ball of radius $r > 0$ centered at a point $x \in X$. We write $\bar{B}_{X}(x,r) = \{y \in X:d(x,y) \leq r\}$ for the closed ball of radius $r > 0$ centered at $x$. We note that the inclusion $\overline{B_{X}(x,r)} \subset \bar{B}_{X}(x,r)$ of the closure of the open ball into the closed ball can be strict in general. By convention all balls $B \subset X$ are considered to have a fixed center and radius, even though it may be the case that we have $B_{X}(x,r) = B_{X}(x',r')$ as sets for some $x \neq x'$, $r \neq r'$. All balls $B \subset X$ are also considered to be open balls unless otherwise specified.  We will write $r(B)$ for the radius of a ball $B$. For a ball $B = B_{X}(x,r)$ in $X$ and a constant $c > 0$ we write $cB = B_{X}(x,cr)$ for the corresponding ball with radius scaled by $c$. For a subset $E \subset  X$ we write $\diam(E) = \sup\{d(x,y):x,y \in E\}$ for the diameter of $E$ and write $\dist(x,E) = \inf\{d(x,y):y \in E\}$ for the infimal distance of a point $x \in X$ to $E$. 

Let $f:(X,d) \rightarrow (X',d')$ be a map between metric spaces. We say that $f$ is \emph{isometric} if $d'(f(x),f(y)) = d(x,y)$ for $x$, $y\in X$. We recall that a curve $\gamma: I \rightarrow X$ is a \emph{geodesic} if it is an isometric mapping of the interval $I \subset \R$ into $X$.  We say that $X$ is \emph{geodesic} if any two points in $X$ can be joined by a geodesic. A \emph{geodesic triangle} $\Delta$ in $X$ consists of three points $x,y,z \in X$ together with geodesics joining these points to one another. Writing $\Delta = \gamma_{1} \cup \gamma_{2} \cup \gamma_{3}$ as a union of its edges, we say that $\Delta$ is \emph{$\delta$-thin} for a given $\delta \geq 0$ if for each point $p \in \gamma_{i}$, $i =1,2,3$, there is a point $q \in \gamma_{j}$ with $d(p,q) \leq \delta$ and $i \neq j$. A geodesic metric space $X$ is \emph{Gromov hyperbolic} if there is a $\delta \geq 0$ such that all geodesic triangles in $X$ are $\delta$-thin; in this case we will also say that $X$ is \emph{$\delta$-hyperbolic}. When considering Gromov hyperbolic spaces $X$ we will usually use the generic distance notation $|xy|:=d(x,y)$ for the distance between $x$ and $y$ in $X$ and the generic notation $xy$ for a geodesic connecting two points $x,y \in X$, even when this geodesic is not unique.

A metric space $(X,d)$ is \emph{proper} if its closed balls are compact. The Gromov boundary $\p X$ of a proper geodesic $\delta$-hyperbolic space $X$ is defined to be the collection of all geodesic rays $\gamma: [0,\infty) \rightarrow X$ up to the equivalence relation of two rays being equivalent if they are at a bounded distance from one another. We will often refer to the point $\omega \in \p X$ corresponding to a geodesic ray $\gamma$ as the \emph{endpoint} of $\gamma$. Using the Arzela-Ascoli theorem it is easy to see in a proper geodesic $\delta$-hyperbolic space that for any points $x,y \in X \cup \p X$ there is a geodesic $\gamma$ joining $x$ to $y$. We will continue to write $xy$ for any such choice of geodesic joining $x$ to $y$. We will allow our geodesic triangles $\Delta$ to have vertices on $\p X$, in which case we will still write $\Delta = xyz$ if $\Delta$ has vertices $x,y,z$. 


As in our previous work \cite{Bu20}, we will use the notation $\p X$ for the Gromov boundary of $X$ even though it conflicts with the notation $\p \Omega = \bar{\Omega}\backslash \Omega$ for the metric boundary of a metric space $(\Omega,d)$ inside its completion $\bar{\Omega}$. Since we always assume that $X$ is proper we will always have $\bar{X} = X$, so the metric boundary of $X$ will always be trivial. Thus there will be no ambiguity in using $\p X$ for the Gromov boundary as well.


For $x,y,z \in X$ the \emph{Gromov product} of $x$ and $y$ based at $z$ is defined by
\begin{equation}\label{Gromov product}
(x|y)_{z} = \frac{1}{2}(|xz|+|yz|-|xy|). 
\end{equation} 
We can also take the basepoint of the Gromov product to be any function $b \in \hat{\mathcal{B}}(X)$. For $b \in \hat{\mathcal{B}}(X)$ the Gromov product based at $b$ is defined by 
\begin{equation}\label{Gromov Busemann product}
(x|y)_{b} = \frac{1}{2}(b(x) + b(y) - |xy|). 
\end{equation}
For $b \in \mathcal{D}(X)$, $b(x) = d(x,z)+s$ this reduces to the notion of Gromov product in \eqref{Gromov product}, as we have $(x|y)_{b} = (x|y)_{z} + s$.

We now consider an incomplete metric space $(\Omega,d)$ and write $\p \Omega = \bar{\Omega}\backslash \Omega$ for the metric boundary of $\Omega$ in its completion $\bar{\Omega}$. We write $d_{\Omega}(x) :=\dist(x,\p \Omega)$ for the distance of a point $x \in \Omega$ to the boundary $\p \Omega$. An important observation that we will use without comment is that $d_{\Omega}$ defines a $1$-Lipschitz function on $\Omega$, i.e., for $x,y \in \Omega$ we have 
\[
|d_{\Omega}(x)-d_{\Omega}(y)| \leq d(x,y). 
\]
For a curve $\gamma: I \rightarrow \Omega$ we write $\ell(\gamma)$ for the length of $\gamma$ and say that $\gamma$ is \emph{rectifiable} if $\ell(\gamma) < \infty$. For an interval $I \subset \R$ and $t \in I$ we write $I_{\leq t} = \{s \in I: s\leq t\}$ and $I_{\geq t} = \{s \in I: s\geq t\}$. For a rectifiable curve $\gamma:I \rightarrow \Omega$ we write $\gamma_{-},\gamma_{+} \in \bar{\Omega}$ for the endpoints of $\gamma$; writing $t_{-} \in [-\infty,\infty)$ and $t_{+} \in (-\infty,\infty]$ for the endpoints of $I$, these are defined by the limits $\gamma(t_{-}) = \lim_{t \rightarrow t_{-}} \gamma(t)$ and $\gamma(t_{+}) = \lim_{t \rightarrow t_{+}} \gamma(t)$ in $\bar{\Omega}$ which exist because $\ell(\gamma) < \infty$. 

\begin{defn}\label{def:uniform}For a constant $A \geq 1$ and an interval $I \subset \R$, a curve $\gamma: I \rightarrow \Omega$ is \emph{$A$-uniform} if 
\begin{equation}\label{uniform one}
\ell(\gamma) \leq Ad(\gamma_{-},\gamma_{+}),
\end{equation}
and if for every $t \in I$ we have
\begin{equation}\label{uniform two}
\min\{\ell(\gamma|_{I_{\leq t}}),\ell(\gamma|_{I_{\geq t}})\} \leq A d_{\Omega}(\gamma(t)). 
\end{equation}
The metric space $\Omega$ is \emph{$A$-uniform} if it is locally compact and if any two points in $\Omega$ can be joined by an $A$-uniform curve. 
\end{defn}

We extend Definition \eqref{def:uniform} to the case of non-rectifiable curves $\gamma: I \rightarrow \Omega$ by replacing \eqref{uniform one} with the condition that $d(\gamma(s),\gamma(t)) \rightarrow \infty$ as $s \rightarrow t_{-}$ and $t \rightarrow t_{+}$. We keep the requirement \eqref{uniform two} the same. Observe that with this extended definition the inequality \eqref{uniform two} implies that an $A$-uniform curve $\gamma$ is always \emph{locally rectifiable}, meaning that each compact subcurve of $\gamma$ is rectifiable. We note that it is easily verified from the definitions that the property of a curve $\gamma$ being $A$-uniform is independent of the choice of parametrization of $\gamma$.

Now let $X$ be a proper geodesic $\delta$-hyperbolic space. We define $X$ to be \emph{$K$-roughly starlike} from a point $z \in X$ if for each $x \in X$ there is a geodesic ray $\gamma:[0,\infty) \rightarrow X$ such that $\dist(x,\gamma) \leq K$. Similarly for $\omega \in \p X$ we define $X$ to be $K$-roughly starlike from $\omega$ if for each $x \in X$ there is a geodesic line $\gamma: \R \rightarrow X$ with $\gamma|_{(-\infty,0]} \in \omega$ and $\dist(x,\gamma) \leq K$. When $\p X$ contains at least two points $K$-rough starlikeness from any point $x \in X \cup \p X$ implies $K'$-rough starlikeness from \emph{all} points of $X \cup \p X$ for a constant $K' \geq 0$ by \cite[Proposition 1.13]{Bu21}. We also note that rough starlikeness from $\omega$ immediately implies that $\p_{\omega} X \neq \emptyset$. 

We fix a function $b \in \hat{\mathcal{B}}(X)$ with basepoint $\omega \in X \cup \p X$ and let $\e > 0$ be such that the density $\rho_{\e}(x) = e^{- \e b(x)}$ is a GH-density on $X$ with constant $M$. Since $b$ is $1$-Lipschitz we have the \emph{Harnack type inequality} for $x,y \in X$,
\begin{equation}\label{Harnack}
e^{-\e|xy|} \leq \frac{\rho_{\e}(x)}{\rho_{\e}(y)} \leq e^{\e|xy|}.
\end{equation}


We write $X_{\e} = X_{\e,b}$ for the conformal deformation of $X$ with conformal factor $\rho$ and write $d_{\e} = d_{\e,b}$ for the resulting distance on $X_{\e}$. We write $\ell_{\e}(\gamma) :=\ell_{\e,b}(\gamma)$ for the lengths of curves measured in the metric $d_{\e}$ and $\ell(\gamma)$ for the lengths of curves measured in $X$. The properness of $X$ implies that $X_{\e}$ is locally compact.  By \cite[Theorem 1.4]{Bu20} the metric space $X_{\e}$ is incomplete and geodesics in $X$ are $A$-uniform curves in $X_{\e}$.  In particular the metric space $(X_{\e},d_{\e})$ is $A$-uniform. Furthermore the space $X_{\e}$ is bounded if and only if $b \in \mathcal{D}(X)$. The proof of \cite[Theorem 1.4]{Bu20} shows that when $b \in \mathcal{D}(X)$ all geodesics in $X$ have finite length in $X_{\e}$, while in the case $b \in \mathcal{B}(X)$ geodesics have finite length if and only if they do not have the basepoint $\omega$ of $b$ as an endpoint. For $x \in \bar{X}_{\e}$ we write $B_{\e}(x,r)$ for the open ball of radius $r > 0$ centered at $x$ in the metric $d_{\e}$ on $\bar{X}_{\e}$, and for $x \in X$ we write $B_{X}(x,r)$ for the open ball of radius $r$ centered at $x$ in $X$. 

For $x \in X_{\e}$ write $d_{\e}(x) = d_{X_{\e}}(x)$ for the distance to the metric boundary $\p X_{\e}$ of $X_{\e}$. By \cite[Theorem 1.6]{Bu20} there is a canonical identification $\varphi_{\e}: \p_{\omega} X \rightarrow \p X_{\e}$ of the Gromov boundary of $X$ relative to $\omega$ and the metric boundary $\p X_{\e}$ of $X_{\e}$; we recall that $\p_{\omega}X = \p X$ if $\omega \in X$ and $\p_{\omega}X = \p X \backslash \{\omega\}$ if $\omega \in \p X$. The correspondence is given by showing that any sequence $\{x_{n}\}$ in $X$ converging to a point $\xi \in \p_{\omega} X$ is a Cauchy sequence in $X_{\e}$ converging to a point of $\p X_{\e}$. 


The local compactness of $X_{\e}$ implies by the Arzela-Ascoli theorem that, for a given $x,y \in X$, a minimizing curve $\gamma$ for the right side of \eqref{define distance} always exists. It is easy to see that such a curve must be a geodesic in $X_{\e}$, from which we conclude that $X_{\e}$ is always geodesic. By \cite[Proposition 2.20]{BHK} the completion $\bar{X}_{\e}$ of $X_{\e}$ is proper, and in particular is also locally compact. A second application of Arzela-Ascoli then shows that $\bar{X}_{\e}$ is also geodesic.

We collect here two important quantitative results regarding the uniformization $X_{\e}$ from our previous work \cite{Bu20}. The standing assumptions for the rest of this section are that $X$ is a proper geodesic $\delta$-hyperbolic space with a given $b \in \hat{\mathcal{B}}(X)$  such that $X$ is $K$-roughly starlike from the basepoint $\omega$ of $b$, and that for a given $\e > 0$ the density $\rho_{\e}(x) = e^{-\e b(x)}$ on $X$ is a GH-density with constant $M$. All implied constants will depend only on $\delta$, $K$, $\e$, and $M$.  

\begin{lem}\label{lem:estimate both}\cite[Lemma 4.7]{Bu20}
For $x,y \in X$ we have
\begin{equation}\label{estimate both}
d_{\e}(x,y) \asymp e^{-\e (x|y)_{b}}\min\{1,|xy|\}.
\end{equation}  
\end{lem}

\begin{lem}\label{compute distance}\cite[Lemma 4.15]{Bu20}
For $x \in X$ we have
\begin{equation}\label{compute distance inequality}
d_{\e}(x) \asymp \rho_{\e}(x). 
\end{equation}
\end{lem}

Lemmas \ref{lem:estimate both} and \ref{compute distance} are stated for $b \in \mathcal{B}(X)$ in \cite{Bu20}, however as noted in \cite[Remark 4.24]{Bu20} the estimates for $b \in \mathcal{D}(X)$ can be deduced from the estimates for $b \in \mathcal{B}(X)$ by attaching a ray to $X$ at the basepoint of a given $b \in \mathcal{D}(X)$.

We conclude this section by adapting two key claims from \cite{BBS20} to our setting. The first claim adapts \cite[Theorem 2.10]{BBS20}. The proof is essentially the same. 

\begin{lem}\label{sub inclusion}
There is a constant  $C_{*} = C_{*}(\delta,K,\e,M) \geq 1$ such that for any $x \in X$ and any $0 < r \leq \frac{1}{2}d_{\e}(x)$ we have the inclusions,
\begin{equation}\label{sub inclusion equation}
B_{X}\left(x,\frac{C_{*}^{-1}r}{\rho_{\e}(x)}\right) \subset B_{\e}(x,r) \subset B_{X}\left(x,\frac{C_{*}r}{\rho_{\e}(x)}\right).
\end{equation}
\end{lem}

\begin{proof}
Let $y \in B_{X}(x,C^{-1}_{*}r/\rho_{\e}(x))$, for a constant $C_{*} \geq 1$ to be determined. Let $\gamma$ be a geodesic in $X$ joining $x$ to $y$ and let $z \in \gamma$. Then, since $r \leq \frac{1}{2}d_{\e}(x)$, we have by Lemma \ref{compute distance},
\[
|xz| \leq \frac{C^{-1}_{*}d_{\e}(x)}{2\rho_{\e}(x)} \leq C_{*}^{-1}C,
\]
with $C = C(\delta,K,\e,M) \geq 1$. This then implies by the Harnack inequality \eqref{Harnack},
\[
\rho_{\e}(z) \asymp_{e^{C_{*}^{-1}C \e}} \rho_{\e}(x).
\]
Choosing $C_{*}$ large enough that $e^{C_{*}^{-1}C \e} < 2$, we then obtain that
\[
\rho_{\e}(z) \asymp_{2} \rho_{\e}(x),
\]
for $z \in \gamma$. We conclude that
\begin{align*}
d_{\e}(x,y) &\leq \int_{\gamma}\rho_{\e}\,ds \\
&\leq 2\rho_{\e}(x)|xy| \\
&\leq 2C_{*}^{-1}r \\
&< r,
\end{align*}
provided we take $C_{*} > 2$. This gives the inclusion on the left side of \eqref{sub inclusion equation}.

For the inclusion on the right side of \eqref{sub inclusion equation}, let $y \in B_{\e}(x,r)$ and let $\gamma_{\e}$ be a geodesic in $X_{\e}$ connecting $x$ to $y$. For $z \in \gamma_{\e}$ we then have $z \in B_{\e}(x,r)$ and therefore $d_{\e}(z) \geq \frac{1}{2} d_{\e}(x)$  by the triangle inequality since $r \leq \frac{1}{2} d_{\e}(x)$. Applying Lemma \ref{compute distance}, we then have
\begin{align*}
\rho_{\e}(z) &\geq C^{-1}d_{\e}(z) \\
&\geq \frac{1}{2}C^{-1}d_{\e}(x) \\
&\geq C^{-1}\rho_{\e}(x),
\end{align*}
for a constant $C = C(\delta,K,\e,M) \geq 1$. Using this we conclude that
\[
r > d_{\e}(x,y) = \int_{\gamma_{\e}}\rho_{\e}\,ds \geq C^{-1}\rho_{\e}(x)|xy|,
\]
since $\ell(\gamma_{\e}) \geq |xy|$. Choosing $C_{*}$ to be greater than the constant $C$ on the right side of this inequality, we then conclude that
\[
|xy| < \frac{C_{*}r}{\rho_{\e}(x)},
\]
which gives the right side inclusion in \eqref{sub inclusion equation}. 
\end{proof}

Following \cite{BBS20}, the balls $B_{\e}(x,r)$ for $x \in X_{\e}$, $0 < r \leq \frac{1}{2}d_{\e}(x)$ will often be referred to as \emph{subWhitney balls}.

The second claim adapts \cite[Lemma 4.8]{BBS20} to our setting. The proof given in \cite{BBS20} strongly relies on the uniformization $X_{\e}$ being bounded in their setting, so when $b \in \mathcal{B}(X)$ we will have to take an approach that is somewhat different. 

\begin{lem}\label{choosing center}
There is a constant $\kappa_{0} = \kappa_{0}(\delta,K,\e,M)$ such that for every $x \in \bar{X}_{\e}$ and every $0 < r \leq 2\,\diam\, X_{\e}$ we can find a ball $B_{\e}(z,\kappa_{0} r) \subset B_{\e}(x,r)$ with $d_{\e}(z) \geq 2\kappa_{0} r$. 
\end{lem}

\begin{proof}
The claim for $b \in \mathcal{D}(X)$ of the form $b_{z}(x) = d(x,z)$ for some $z \in X$ follows from repeating the proof of \cite[Lemma 4.8]{BBS20} in our setting. For $b \in \mathcal{D}(X)$ of the form $b(x) = d(x,z)+s$ for some $z \in X$, $s \in \R$, the claim then follows by observing that $X_{\e} = e^{-s}X_{\e,z}$, i.e., $X_{\e}$ is obtained by scaling by a factor of $e^{-s}$ the metric on the conformal deformation of $X$ by $\rho_{\e,z}(x) = e^{-\e|xz|}$. We can thus assume that $b \in \mathcal{B}(X)$ with basepoint $\omega \in \p X$, which implies that $\diam \, X_{\e} = \infty$.

Let $x \in \bar{X}_{\e}$ and $r > 0$ be given. The function $y \rightarrow d_{\e}(y)$ on $X_{\e}$ is continuous, positive, unbounded (since $X_{\e}$ is unbounded) and takes values arbitrarily close to $0$ since $d_{\e}(x_{n}) \rightarrow 0$ for any sequence of points $\{x_{n}\}$ in $X_{\e}$ converging to a point of $\p X_{\e}$. Since $X_{\e}$ is connected we can then conclude by the intermediate value theorem that $d_{\e}(X_{\e}) = (0,\infty)$, i.e., for any $r > 0$ we can find a point $z_{0} \in X_{\e}$ such that $d_{\e}(z_{0}) = r$. For our given $r > 0$ we fix such a point $z_{0}$ and let $\sigma$ be a geodesic in $X$ joining $x$ to $z_{0}$; recall that we can consider points $x \in \p X_{\e}$ as points of $\p_{\omega}X$ through the identification $\p X_{\e} \cong \p_{\omega}X$. Then $\sigma$ is an $A$-uniform curve in $X_{\e}$ with $A = A(\delta,K,\e,M) \geq 1$. Since $\sigma$ does not have $\omega$ as an endpoint, it has finite length $\ell_{\e}(\sigma) \leq Ad_{\e}(x,z_{0})$ in $X_{\e}$. We parametrize $\sigma$ by $d_{\e}$-arclength and orient it from $x$ to $z_{0}$. 

We first assume that $\ell_{\e}(\sigma) \geq \frac{2}{3}r$. In this case we set $z = \sigma(\frac{1}{3}r)$. Then since $\sigma$ is $A$-uniform we have $d_{\e}(z) \geq \frac{r}{3A}$ and 
\[
B_{\e}\left(z,\frac{r}{6A}\right) \subset B_{\e}\left(x,\frac{r}{3}+\frac{r}{6A}\right) \subset B_{\e}(x,r). 
\]
So in this case we can use any $\kappa \leq \frac{1}{6A}$. 

Now consider the case in which $\ell_{\e}(\sigma) < \frac{2}{3}r$. We then set $z = z_{0}$ and observe that 
\[
B_{\e}\left(z_{0},\frac{r}{3}\right) \subset B_{\e}\left(x,\ell_{\e}(\sigma)+ \frac{r}{3}\right) \subset B_{\e}(x,r). 
\]
By construction we have $d_{\e}(z_{0}) = r$. Thus in this case any $\kappa \leq \frac{1}{3}$ will work. By combining these two cases we can then set $\kappa_{0} = \frac{1}{6A}$, noting that $A \geq 1$. 
\end{proof}

The conclusion of Lemma \ref{choosing center} is closely related to the \emph{corkscrew condition} for domains in metric spaces. See \cite[Definition 2.4]{BJS07}.

\section{Doubling for uniformized measures}\label{sec:doubling} 


In this section we will prove Theorem \ref{doubling threshold} and lay some of the groundwork for proving our other theorems. We will frequently make use of the following consequence of the doubling estimate \eqref{define doubling measure} for a metric measure space $(X,d,\mu)$: if $\mu$ is doubling on balls of radius at most $R_{0}$ with constant $C_{\mu}$ and $0 < r \leq R \leq R_{0}$ then 
\begin{equation}\label{define doubling consequence}
\mu(B_{X}(x,R)) \asymp_{C} \mu(B_{X}(x,r)),
\end{equation}
with constant $C$ depending only on $C_{\mu}$ and the ratio $R/r$. This estimate follows by iterating the estimate \eqref{define doubling measure} and noting that $\mu(B_{X}(x,R)) \geq \mu(B_{X}(x,r))$ since $B_{X}(x,r) \subset B_{X}(x,R)$. 

We will require the following proposition from \cite{BBS20}, which is stated there in a more general form. 

\begin{prop}\label{enlarge doubling}\cite[Proposition 3.2]{BBS20} Let $(X,d)$ be a geodesic metric  space and let $\mu$ be a Borel measure on $X$ that is doubling on balls of radius at most $R_{0}$ with doubling constant $C_{\mu}$. Then for any $R_{1} > 0$ the measure $\mu$ is doubling on balls of radius at most $R_{1}$, with doubling constant depending only on $R_{1}/R_{0}$ and $C_{\mu}$. 
\end{prop}

Thus if $\mu$ is doubling on balls of radius at most $R_{0}$ then given any $R_{1} > 0$ we can assume that $\mu$ is also doubling on balls of radius at most $R_{1}$, at the cost of increasing the uniform local doubling constant of $\mu$ by an amount depending only on $R_{1}/R_{0}$ and $C_{\mu}$. 


We now describe the setting of this section. We begin with a proper geodesic $\delta$-hyperbolic $X$ together with a function $b \in \hat{\mathcal{B}}(X)$ with basepoint $\omega$ such that $X$ is $K$-roughly starlike from $\omega$. We let $\e > 0$ be such that the associated density $\rho_{\e}$ is a GH-density for $X$ with constant $M$. As in the previous section we write $X_{\e}$ for the uniformization of $X$, $d_{\e}$ for the distance on $X_{\e}$, etc. We let $\mu$ be a Borel regular measure on $X$ such that there is an $R_{0} > 0$ for which $\mu$ is doubling on balls of radius at most $R_{0}$ with doubling constant $C_{\mu}$. For a given $\beta > 0$ we then define the uniformized measure $\mu_{\beta} = \mu_{\beta,b}$ on $\bar{X}_{\e}$ as in \eqref{define beta}. 



In the claims in the rest of this section all implicit constants will depend only on $\delta$, $K$, $\e$, $M$, $\beta$, $R_{0}$, and $C_{\mu}$. We will refer to this collection of seven parameters as the \emph{data}. We will refer to the specific parameters $\delta$, $K$, $\e$, $M$, and $\beta$ as the \emph{uniformization data} and say that a constant depends only on the uniformization data if it depends only on these five parameters.  At several points we will need to increase the radius $R_{0}$ by an amount depending only on the uniformization data in order to ensure that $\mu$ is doubling at a larger scale using Proposition \ref{enlarge doubling}. When we do this we will also need to increase $C_{\mu}$ by a corresponding amount depending only on the uniformization data and the local doubling constant $C_{\mu}$ for $\mu$. 

\begin{rem}\label{exclude beta}
We will also often refer to just the four parameters $\delta$, $K$, $\e$, and $M$ as the uniformization data. It will be clear from context when $\beta$ can and cannot be excluded from the list. 
\end{rem}

The first part of this section will be devoted to proving the following technical criterion for $\mu_{\beta}$ to be doubling on $\bar{X}_{\e}$. Throughout this section we let $\kappa_{0} = \kappa_{0}(\delta,K,\e,M)$ be defined as in Lemma \ref{choosing center} and set $\kappa_{1} = \kappa_{0}/10$. 

\begin{prop}\label{global doubling}
Suppose that there is a constant $C_{0} \geq 1$ such that for any $\xi \in \p X_{\e}$, $r > 0$, and $z \in X$ we have  that whenever $B_{\e}(z,\kappa_{1} r) \subset B_{\e}(\xi,r)$ and $d_{\e}(z) \geq 2\kappa_{1} r$,
\begin{equation}\label{controlled upper}
\mu_{\beta}(B_{\e}(\xi,r)) \leq C_{0}r^{\beta/\e}\mu(B_{X}(z,R_{0})).
\end{equation}
Then $\mu_{\beta}$ is doubling on $\bar{X}_{\e}$ with doubling constant $C_{\mu_{\beta}}$ depending only on the data and $C_{0}$. 
\end{prop}

We have formulated Proposition \ref{global doubling} in the manner that is most convenient for us to verify in practice, however this comes at the cost of obscuring the connection of the inequality \eqref{controlled upper} to the doubling property for $\mu_{\beta}$. In order to prove Proposition \ref{global doubling} we will need a series of lemmas that establish this connection. Our first claim corresponds to \cite[Lemma 4.5]{BBS20}. It provides an estimate on the measure of subWhitney balls in $X_{\e}$.  


\begin{lem}\label{sub measure}
Let $x \in X$ and $0 < r \leq \frac{1}{2}d_{\e}(x)$. Then 
\[
\mu_{\beta}(B_{\e}(x,r)) \asymp \rho_{\beta}(x)\mu\left(B_{X}\left(x, \frac{r}{\rho_{\e}(x)}\right)\right),
\]
with comparison constant depending only on the data.
\end{lem}

\begin{proof}
By Lemma \ref{compute distance} we have for all $y \in B_{\e}(x,r)$, 
\begin{equation}\label{comparison chain}
\rho_{\beta}(y) = \rho_{\e}(y)^{\beta/\e} \asymp d_{\e}(y)^{\beta/\e} \asymp d_{\e}(x)^{\beta/\e} \asymp \rho_{\beta}(x),
\end{equation}
with the comparison $d_{\e}(y) \asymp_{2} d_{\e}(x)$ following from the condition on $r$. Applying Lemma \ref{sub inclusion} and the chain of comparisons \eqref{comparison chain}, we conclude that
\[
\mu_{\beta}(B_{\e}(x,r)) \asymp \rho_{\beta }(x)\mu(B_{\e}(x,r)) \lesssim \rho_{\beta }(x) \mu\left(B_{X}\left(x,\frac{C_{*}r}{\rho_{\e}(x)}\right)\right),
\]
with $C_{*} = C_{*}(\delta,K,\e,M)$ being the constant from Lemma \ref{sub inclusion}. A similar argument using the other inclusion from Lemma \ref{sub inclusion} shows that 
\[
\mu_{\beta}(B_{\e}(x,r)) \gtrsim \rho_{\beta}(x) \mu\left(B_{X}\left(x,\frac{C_{*}^{-1}r}{\rho_{\e}(x)}\right)\right).
\]
We thus conclude that
\begin{equation}\label{proto sub measure}
\rho_{\beta}(x) \mu\left(B_{X}\left(x,\frac{C_{*}^{-1}r}{\rho_{\e}(x)}\right)\right)\lesssim \mu_{\beta}(B_{\e}(x,r)) \lesssim \rho_{\beta}(x) \mu\left(B_{X}\left(x,\frac{C_{*}r}{\rho_{\e}(x)}\right)\right)
\end{equation}
The condition on $r$ implies that 
\begin{equation}\label{ratio inequality}
\frac{r}{\rho_{\e}(x)} \leq \frac{1}{2}\frac{d_{\e}(x)}{\rho_{\e}(x)} \leq C,
\end{equation}
with $C$ depending only on the uniformization data by Lemma \ref{compute distance}. By Proposition \ref{enlarge doubling} we can, at the cost of increasing the local doubling constant $C_{\mu}$ of $\mu$ by an amount depending only on the data, assume that $R_{0} > CC_{*}$ for the constant $C$ in inequality \eqref{ratio inequality} and the constant $C_{*}$ in Lemma \ref{sub inclusion}. Then the comparison \eqref{define doubling consequence} allows us to conclude that
\[ 
 \mu\left(B_{X}\left(x,\frac{C_{*}^{-1}r}{\rho_{\e}(x)}\right)\right) \asymp \mu\left(B_{X}\left(x,\frac{r}{\rho_{\e}(x)}\right)\right) \asymp \mu\left(B_{X}\left(x,\frac{C_{*}r}{\rho_{\e}(x)}\right)\right).
\]
Combining this comparison with inequality \eqref{proto sub measure} proves the lemma. 
\end{proof}

By combining Lemma \ref{sub measure} with Lemma \ref{choosing center} we obtain the following estimate for $\mu_{\beta}(B_{\e}(x,r))$ when $0 < r \leq \frac{1}{2}d_{\e}(x)$. We recall that $\kappa_{1} = \kappa_{0}/10$, where $\kappa_{0}$ is defined as in Lemma \ref{choosing center}. 

\begin{lem}\label{choosing center measure}
Let $x \in X$ and $0 < r \leq \frac{1}{2}d_{\e}(x)$. Let $z \in X$ be given such that $B_{\e}(z,\kappa_{1} r) \subset B_{\e}(x,r)$ and $d_{\e}(z) \geq 2\kappa_{1} r$. Then
\[
\mu_{\beta}(B_{\e}(x,r)) \asymp \mu_{\beta}(B_{\e}(z,\kappa_{1} r)),
\]
with comparison constants depending only on the data. 
\end{lem}

\begin{proof}
By Lemmas \ref{compute distance} and \ref{sub inclusion} we have 
\begin{equation}\label{pre radius control}
|xz| \leq \frac{C_{*}r}{\rho_{\e}(x)} \leq \frac{C_{*}d_{\e}(x)}{2\rho_{\e}(x)} \lesssim 1,
\end{equation}
with implied constant depending only on the uniformization data, where $C_{*}$ is the constant from Lemma \ref{sub inclusion}. Since $z \in B_{\e}(x,r)$ and $r \leq \frac{1}{2}d_{\e}(x)$, we conclude that we have $d_{\e}(z) \asymp_{2} d_{\e}(x)$. We thus obtain from Lemma \ref{compute distance} that $\rho_{\e}(z) \asymp \rho_{\e}(x)$ with comparison constant depending only on the uniformization data. Since $d_{\e}(z) \geq 2\kappa_{1} r$, we have by Lemma \ref{compute distance} that
\begin{equation}\label{radius control}
1 \gtrsim \frac{\kappa_{1} r}{\rho_{\e}(z)} \asymp \frac{r}{\rho_{\e}(z)} \asymp \frac{r}{\rho_{\e}(x)} \asymp \frac{C_{*}r}{\rho_{\e}(x)},
\end{equation}
with all implied constants depending only on the uniformization data, since $\kappa_{1}$ depends only on the uniformization data. We can thus apply Proposition \ref{enlarge doubling} to conclude that we can assume that $\mu$ is doubling on balls of radius at most any of the terms appearing in \eqref{radius control}, at the cost of increasing the local doubling constant of $\mu$ by an amount depending only on the data. It follows that
\begin{align*}
\mu\left(B_{X}\left(z,\frac{\kappa_{1} r}{\rho_{\e}(z)}\right)\right) &\asymp \mu\left(B_{X}\left(z,\frac{r}{\rho_{\e}(z)}\right)\right) \\
&\asymp \mu\left(B_{X}\left(z,\frac{r}{\rho_{\e}(x)}\right)\right)\\
&\asymp \mu\left(B_{X}\left(x,\frac{C_{*}r}{\rho_{\e}(x)}\right)\right) \\
&\asymp \mu\left(B_{X}\left(x,\frac{r}{\rho_{\e}(x)}\right)\right)
\end{align*}
with implied constants depending only on the data. The third comparison above follows from the fact that $z \in B_{X}\left(x,\frac{C_{*}r}{\rho_{\e}(x)}\right)$ by \eqref{pre radius control}. Since the comparison $\rho_{\beta}(z) \asymp \rho_{\beta}(x)$ follows from the comparison $\rho_{\e}(z) \asymp \rho_{\e}(x)$ (with comparison constants depending only on the uniformization data), applying Lemma \ref{sub measure} to $B_{\e}(z,\kappa r)$ and $B_{\e}(x,r)$ (note that $\kappa_{1} r \leq \frac{1}{2}d_{\e}(z)$ by assumption)  then gives 
\[
\mu_{\beta}(B_{\e}(z,\kappa_{1} r)) \asymp \mu_{\beta}(B_{\e}(x,r)),
\]
with comparison constants depending only on the data.
\end{proof}

Our final lemma estimates the right side of inequality \eqref{controlled upper} in terms of $\mu_{\beta}(B_{\e}(z,\kappa_{1}r))$. The reason for choosing $5r$ as the upper bound for $d_{\e}(z)$ will be clear in the proof of Proposition \ref{global doubling}. 

\begin{lem}\label{estimate upper}
Let $z \in X$ and $r > 0$  be such that $2\kappa_{1} r \leq d_{\e}(z) < 5r$. Then
\begin{equation}\label{comparison upper}
\mu_{\beta}(B_{\e}(z,\kappa_{1} r)) \asymp r^{\beta/\e}\mu(B_{X}(z,R_{0})),
\end{equation}
with comparison constant depending only on the data. 
\end{lem}

\begin{proof}
The assumptions imply that $d_{\e}(z) \asymp r$, hence $\rho_{\beta}(z) \asymp r^{\beta/\e}$ by Lemma \ref{compute distance}, with comparison constants depending only on the uniformization data since $\kappa_{1}$ depends only on the uniformization data. Thus by Lemma \ref{sub measure} we have
\[
\mu_{\beta}(B_{\e}(z,\kappa_{1} r)) \asymp r^{\beta/\e}\mu\left(B_{X}\left(z, \frac{\kappa_{1} r}{\rho_{\e}(z)}\right)\right),
\]
with comparison constant depending only on the data. Since $\rho_{\e}(z) \asymp d_{\e}(z) \asymp r$, we have $\frac{\kappa_{1} r}{\rho_{\e}(z)} \asymp_{C'} 1$ for a constant $C'$ depending only on the uniformization data. Using Proposition \ref{enlarge doubling} we can assume that $\mu$ is doubling on balls of radius at most $C'R_{0}$, at the cost of increasing the doubling constant by an amount depending only on the data. From this we conclude that the comparison \eqref{comparison upper} holds. 
\end{proof}

We can now prove Proposition \ref{global doubling}.

\begin{proof}[Proof of Proposition \ref{global doubling}]
We split the proof of the doubling property for $\mu_{\beta}$ into two cases depending on the center $x \in \bar{X}_{\e}$ of the ball. The first case is that in which $0 < r \leq \frac{1}{4}d_{\e}(x)$, which implies in particular that $x \in X_{\e}$. Then we can apply Lemma \ref{sub measure} to both $B_{\e}(x,r)$ and $B_{\e}(x,2r)$. We conclude that
\begin{equation}\label{case one}
\mu_{\beta}(B_{\e}(x,r)) \asymp \mu\left(B_{X}\left(x, \frac{r}{\rho_{\e}(x)}\right)\right) \asymp \mu\left(B_{X}\left(x, \frac{2r}{\rho_{\e}(x)}\right)\right) \asymp \mu_{\beta}(B_{\e}(x,2r)),
\end{equation}
with comparison constants depending only on the data. To justify the middle comparison in \eqref{case one}, we observe that since $2r \leq \frac{1}{2}d_{\e}(x)$ we have by Lemma \ref{compute distance} that each of the middle two balls in $X$ in \eqref{case one} on the right side of this inequality have radius at most $C'$ for some constant $C'$ depending only on the uniformization data. By Proposition \ref{enlarge doubling} we can assume that $\mu$ is doubling on balls of radius at most $C'$, at the cost of increasing the doubling constant of $\mu$ by an amount depending only on the data. This gives the desired doubling estimate for the right side of \eqref{case one}. We note that this first case does not require the use of the assumed inequality \eqref{controlled upper}.

The second case is that in which $d_{\e}(x) < 4r$. We can then find a point $\xi \in \p X_{\e}$ such that $B_{\e}(x,r) \subset B_{\e}(\xi,5r)$. We then use Lemma \ref{choosing center} to choose a point $z \in X_{\e}$ such that $B_{\e}(z,\kappa_{0}r) \subset B_{\e}(x,r)$ and $d_{\e}(z) \geq 2\kappa_{0}r$. Then we must have $d_{\e}(z) < 5r$ since $z \in B_{\e}(\xi,5r)$. Since $B_{\e}(z,\kappa_{0}r) \subset B_{\e}(\xi,5r)$ and $\kappa_{0} > 5\kappa_{1}$, we conclude from Lemma \ref{estimate upper} and the assumed inequality \eqref{controlled upper} that 
\begin{equation}\label{case two a}
\mu_{\beta}(B_{\e}(x,r)) \asymp \mu_{\beta}(B_{\e}(z,\kappa_{1} r)),
\end{equation}
with comparison constant depending only on the data and $C_{0}$. Since we also have $B_{\e}(x,2r) \subset B_{\e}(\xi,10r)$ and $\kappa_{0} = 10\kappa_{1}$, the same combination of Lemma \ref{estimate upper} and \eqref{controlled upper} also shows that
\begin{equation}\label{case two b}
\mu_{\beta}(B_{\e}(x,2r)) \asymp \mu_{\beta}(B_{\e}(z,\kappa_{1} r)),
\end{equation}
with comparison constant depending only on the data and $C_{0}$. Combining \eqref{case two a} and \eqref{case two b} gives the desired doubling estimate in this second case.
\end{proof}




We will now prove Theorem \ref{doubling threshold} by showing, in analogy to \cite[Proposition 4.7]{BBS20}, that $\mu_{\beta}$ is always doubling on $\bar{X}_{\e}$ for $\beta$ sufficiently large. We will need the following refinement of Proposition \ref{enlarge doubling}.

\begin{lem}\label{enlarge doubling refined}\cite[Lemma 3.5]{BBS20} Let $(X,d)$ be a geodesic metric space and let $\mu$ be a measure on $X$ that is doubling on balls of radius at most $R_{0}$ with constant $C_{\mu}$. Let $n \in \N$ be a given integer. 
\begin{enumerate}
\item For $x,y \in X$ and $0 < r \leq R_{0}$ satisfying $d(x,y) < nr$, we have
\[
\mu(B_{X}(x,r)) \leq C_{\mu}^{n}\mu(B_{X}(y,r)). 
\]
\item For $0 < r \leq \frac{1}{4}R_{0}$, every ball $B \subset X$ of radius $nr$ can be covered by at most $C^{7(n+4)/6}_{\mu}$ balls of radius $r$.
\end{enumerate}  
\end{lem}

\begin{proof}[Proof of Theorem \ref{doubling threshold}]
We will prove this theorem using the criterion of Proposition \ref{global doubling}. All implied constants will depend only on the data, meaning only on the parameters $\delta$, $K$, $\e$, $M$, $\beta$, $R_{0}$ and $C_{\mu}$, unless otherwise noted.  Let $\xi \in \p X_{\e}$, $z \in X_{\e}$, and $r > 0$ be given such that $B_{\e}(z,\kappa_{1} r) \subset B_{\e}(\xi,r)$ and $d_{\e}(z) \geq 2\kappa_{1} r$, where we recall that $\kappa_{1} = \kappa_{0}/10$ depends only on the uniformization data. We then have $\rho_{\beta}(z) \asymp r^{\beta/\e}$ by the proof of Lemma \ref{estimate upper}. We define for $n \geq 1$, 
\[
A_{n} = \{x \in B_{\e}(\xi,r) \cap X_{\e}: e^{-\e n}r \leq d_{\e}(x) < e^{-\e(n-1)}r\}.
\]
Since $x \in B_{\e}(\xi,r)$ implies that $d_{\e}(x) < r$, we have $B_{\e}(\xi,r) \cap X_{\e} = \bigcup_{n=1}^{\infty} A_{n}$. Since $\mu_{\beta}$ is extended to $\p X_{\e}$ by setting $\mu_{\beta}(\p X_{\e}) = 0$, we conclude that
\[
\mu_{\beta}(B_{\e}(\xi,r)) = \sum_{n=1}^{\infty} \mu_{\beta}(A_{n}). 
\]

For any given $x \in A_{n}$ we either have $|xz| < 1$ or $|xz| \geq 1$. In the second case we use Lemma \ref{lem:estimate both} to obtain
\begin{align*}
e^{\e | xz|} &= \frac{e^{-2\e(x|z)_{b}}}{\rho_{\e}(x)\rho_{\e}(z)} \\
&\asymp \frac{d_{\e}(x,z)^{2}}{d_{\e}(x)d_{\e}(z)} \\
&\leq \frac{(d_{\e}(x,\xi) + d_{\e}(\xi,z))^{2}}{2\kappa_{1} e^{-\e n}r^{2}} \\
&\leq \frac{2e^{\e n}}{\kappa_{1}} \\
&\ls e^{\e n},
\end{align*}
with implied constant depending only on $\delta$, $K$, $\e$, and $M$.  We then conclude that $ |xz| \leq n + c_{0}$, with $c_{0}=c_{0}(\delta,K,\e,M) \geq 0$. Since this inequality trivially holds with $c_{0} = 0$ when $|xz| < 1$, we in fact obtain the inequality  $|xz| \leq n + c_{0}$ in both cases. We then choose $n_{0} = n_{0}(\delta,K,\e,M)$ to be the minimal integer such that $n_{0} \geq c_{0}$. Then for $n \geq n_{0}$ we have $|xz| \leq 2n$ for $x \in A_{n}$. For $1 \leq n \leq n_{0}$ we then have $|xz| \leq 2n_{0}$ and therefore $A_{n} \subset B_{X}(z,2n_{0})$. By Proposition \ref{enlarge doubling} we can then assume that $\mu$ is doubling on balls of radius at most $2n_{0}$ in $X$, at the cost of increasing the doubling constant by an amount depending only on $\delta$, $K$, $\e$, and $M$. We conclude that
\[
\mu(B_{X}(z,2n_{0})) \asymp \mu(B_{X}(z,R_{0})),
\]
with comparison constant depending only on the data. By the Harnack inequality \eqref{Harnack} for $\rho_{\beta}$ we conclude for $x \in B_{X}(z,2n_{0})$ that $\rho_{\beta}(x) \asymp \rho_{\beta}(z) \asymp r^{\beta/\e}$. Putting all of this together, we conclude that
\[
\mu_{\beta}\left(\bigcup_{n=1}^{n_{0}}A_{n}\right) \leq \mu_{\beta}(B_{X}(z,2n_{0})) \ls r^{\beta/\e}\mu(B_{X}(z,R_{0})).
\]

We now consider the case $n > n_{0}$, for which we have $|xz| \leq 2n$ whenever $x \in A_{n}$. We apply Proposition \ref{enlarge doubling} to ensure that $\mu$ is doubling on balls of radius at most $R_{1} = \max\{4R_{0},R_{0}+2\}$. The doubling constant $C'_{\mu}$ for $\mu$ on balls of radius at most $R_{1}$ then depends only on $R_{0}$ and $C_{\mu}$. In particular $C'_{\mu}$ does not depend on $\beta$. Applying (2) of Lemma \ref{enlarge doubling refined}, we cover $A_{n} \subset B_{X}(z,2n)$ with $N_{n} \lesssim e^{\alpha n} $ many balls $B_{n,j}$ of radius $R_{0}$, where 
\[
\alpha = \alpha(R_{0},C_{\mu}) = \frac{7}{6}\log C_{\mu}'.
\]
We set $\beta_{0}:=3\alpha$ and assume that $\beta \geq \beta_{0}$. Note that $\beta_{0} = \beta_{0}(R_{0},C_{\mu})$ depends only on $R_{0}$ and $C_{\mu}$. 

We can clearly assume that each ball $B_{n,j}$ intersects $A_{n}$, from which we conclude that the centers $x_{n,j}$ of the balls $B_{n,j}$ satisfy 
\[
|x_{n,j}z| \leq R_{0} + 2n < R_{1}n,
\]
since $n \geq 1$ and $R_{1} > R_{0}+2$. Applying (1) of Lemma \ref{enlarge doubling refined} then gives that
\[
\mu(B_{n,j}) \leq (C_{\mu}')^{n}\mu(B_{X}(z,R_{1})) \leq e^{\alpha n}\mu(B_{X}(z,R_{1})),
\]
For $x \in A_{n}$ we have, 
\begin{equation}\label{annulus comparison}
\rho_{\beta}(x) \asymp d_{\e}(x)^{\beta/\e} \asymp (e^{-\e n}r)^{\beta/\e}.
\end{equation}
The Harnack inequality \eqref{Harnack} implies that $\rho_{\beta}(y) \asymp \rho_{\beta}(x_{n,j})$ for each $y \in B_{n,j}$ (since each ball $B_{n,j}$ has radius $R_{0}$). Furthermore, since there is some point $y\in A_{n}$ such that $|x_{n,j}y| \leq R_{0}$, it follows from the comparison \eqref{annulus comparison} that $\rho_{\beta}(x_{n,j}) \asymp (e^{-\e n}r)^{\beta/\e}$. Thus we conclude that
\begin{align*}
\mu_{\beta}(B_{n,j}) &\asymp \rho_{\beta}(x_{n,j}) \mu(B_{n,j}) \\
&\lesssim (e^{-\e n} r)^{\beta/\e}\mu(B_{n,j}) \\
&\leq e^{-\beta n} r^{\beta/\e} e^{\alpha n}\mu(B_{X}(z,R_{1})).
\end{align*}
By our restriction $\beta \geq \beta_{0} = 3\alpha$, we conclude that
\[
\mu_{\beta}(B_{n,j}) \ls e^{-2\alpha n}r^{\beta/\e} \mu(B_{X}(z,R_{1})).
\]
It then follows from this inequality and the bound $N_{n} \ls e^{\alpha n}$ that  
\begin{align*}
\mu_{\beta}\left(\bigcup_{n = n_{0}+1}^{\infty}A_{n}\right) &\leq \sum_{n=n_{0}+1}^{\infty} \sum_{j=1}^{N_{n}} \mu_{\beta}(B_{n,j}) \\
&\lesssim r^{\beta/\e} \mu(B_{X}(z,R_{1})) \sum_{n=n_{0}+1}^{\infty} N_{n} e^{- 2\alpha n} \\ 
&\ls r^{\beta/\e} \mu(B_{X}(z,R_{1})) \sum_{n=n_{0}+1}^{\infty} e^{- \alpha n} \\
&\ls r^{\beta/\e}\mu(B_{X}(z,R_{1})),
\end{align*}
with the final inequality following by summing the geometric series. By combining the cases $1 \leq n \leq n_{0}$ and $n > n_{0}$ we conclude that
\[
\mu_{\beta}(B_{\e}(\xi,r)) \ls r^{\beta/\e}\mu(B_{X}(z,R_{1})).
\]
Since $\mu$ is doubling up to the radius $R_{1}$ with doubling constant depending only on the data, we conclude by Proposition \ref{global doubling} that $\mu_{\beta}$ is doubling on $\bar{X}_{\e}$ with constant depending only on the data. 
\end{proof}

We now discuss a setting in which it is possible to obtain sharper estimates for the threshold $\beta_{0}$ above which $\mu_{\beta}$ is doubling. In particular this will allow us to prove the doubling claim in Theorem \ref{Riem doubling}. We will keep the setting of Theorem \ref{doubling threshold} and then assume in addition that we have a cocompact discrete isometric action of a group $\Gamma$ on $X$. Briefly recalling the definitions, the action by $\Gamma$ is \emph{isometric} if each element $g \in \Gamma$ defines an isometry of $X$. It is \emph{cocompact} if there is a compact set $E \subset X$ such that $X = \bigcup_{g \in \Gamma} g(E)$, i.e., if $X$ is  covered by the translates of a compact subset under the action of $\Gamma$. It is \emph{discrete} if for each compact subset $E \subset X$ the number of $g \in \Gamma$ such that $g(E) \cap E \neq \emptyset$ is finite. We will assume in addition that the uniformly locally doubling measure $\mu$ is $\Gamma$-invariant, meaning that $\mu(g^{-1}(E)) = \mu(E)$ for each measurable subset $E \subset X$ and each $g \in \Gamma$. Such measures often arise naturally in the context of the geometry of $X$; for instance if $X$ is a tree with bounded vertex degree and edges of unit length then we can take $\mu$ to be the measure on $X$ induced from the $1$-dimensional Lebesgue measure on the edges. Another case is the setting of Theorem \ref{Riem doubling} when $X$ is the universal cover of a closed Riemannian manifold $M$ with sectional curvatures $\leq -1$, in which case we can take $\mu$ to be the Riemannian volume on $X$. We can assume by Proposition \ref{enlarge doubling} and the cocompactness of the action of $\Gamma$ that the doubling radius $R_{0}$ for $\mu$ is large enough that for each $x \in X$ the translates of $B_{X}(x,R_{0})$ by $\Gamma$ cover $X$. 

For $x \in X$ and $R > 0$ we set
\[
N_{\Gamma}(x,R) = \#\{g \in \Gamma: |xg(x)| \leq R\},
\]
with $\# E$ denoting the cardinality of a set $E$. We consider the critical exponent $h_{X}$ defined by the following limit for a fixed $x \in X$, 
\begin{equation}\label{crit limit}
h_{X} = \limsup_{R \rightarrow \infty} \frac{\log N_{\Gamma}(x,R)}{R}. 
\end{equation}
Standard arguments using the cocompactness of the action of $\Gamma$ show that $h_{X}$ does not depend on the choice of point $x \in X$. We observe that $h_{X}$ can equivalently be thought of as the limit 
\begin{equation}\label{volume limit}
h_{X} = \limsup_{R \rightarrow \infty} \frac{\log \mu(B_{X}(x,R))}{R}, 
\end{equation}
by observing that the translates $g(B_{X}(x,R_{0}))$ for $g \in \Gamma$ will cover  $X$ with bounded overlap by the uniformly local doubling property of $\mu$ and the discreteness of the action of $\Gamma$; for this we can always enlarge the doubling radius to $2R_{0}$ using Proposition \ref{enlarge doubling} to obtain the bounded overlap property. Consequently $\mu(B_{X}(x,R))$ will be comparable to $N_{\Gamma}(x,R)$ when $R$ is large, which shows that the limits \eqref{crit limit} and \eqref{volume limit} are the same. The volume growth entropy \eqref{volume growth} considered in Theorem \ref{Riem doubling} is a special case of the limit \eqref{volume limit}.

It's clear from applying (2) of Lemma \ref{enlarge doubling refined} to the limit \eqref{volume limit} that we have $h_{X} < \infty$. Thus for each $h > h_{X}$ and $x \in X$ we have a constant $C_{h,x} \geq 1$ such that for all $R > 0$, 
\begin{equation}\label{pre exponential}
N_{\Gamma}(x,R) \leq C_{h,x} e^{h R}. 
\end{equation}
The lemma below shows that we can take the constant $C_{h,x}$ to be independent of $x$.

\begin{lem}\label{independence}
For each $h > h_{X}$ there is a constant $C_{h}$ such that we have for all $x \in X$ and $R > 0$,
\begin{equation}\label{exponential crit}
N_{\Gamma}(x,R) \leq C_{h} e^{h R}. 
\end{equation}
\end{lem}

\begin{proof}
Fix $x \in X$ and let $C_{h,x}$ be the constant in \eqref{pre exponential}. Recall that $R_{0} > 0$ was chosen such that the translates of $B_{X}(x,R_{0})$ by $\Gamma$ cover $X$. For each $y \in B_{X}(x,R_{0})$ we have 
\[
N_{\Gamma}(y,R) \leq N_{\Gamma}(x,R+R_{0}) \leq C_{h,x} e^{h(R+R_{0})}. 
\]
This implies that for $y \in B_{X}(x,R_{0})$ we can take $C_{h} = C_{h,x}e^{2hR_{0}}$. It then follows that if $y \in g(B_{X}(x,R_{0}))$ for some $g \in \Gamma$ then 
\[
N_{\Gamma}(y,R) = N_{\Gamma}(g^{-1}(y),R) \leq C_{h} e^{h R}. 
\]
\end{proof}




\begin{prop}\label{crit doubling}
For each $\beta > h_{X}$ the measure $\mu_{\beta}$ on $\bar{X}_{\e}$ is doubling with doubling constant $C_{\mu_{\beta}}$ depending only on the data and the constant $C_{h}$ in \eqref{exponential crit} with $h = (\beta+ h_{X})/2$. 
\end{prop}

\begin{proof}
We follow the outline of the proof of Theorem \ref{doubling threshold} above for a given  $\beta > h_{X}$, but using the estimate \eqref{exponential crit} in place of the use of Lemma \ref{enlarge doubling refined}. As in that proof, we let $\xi \in \p X_{\e}$, $z \in X_{\e}$, and $r > 0$ be given such that $B_{\e}(z,\kappa_{1} r) \subset B_{\e}(\xi,r)$ and $d_{\e}(z) \geq 2\kappa_{1} r$, where we recall that $\kappa_{1} = \kappa_{0}/10$ depends only on the uniformization data, and we will take all implied constants to depend only on the data and the constant $C_{h}$ in Lemma \ref{independence} with $h = (\beta+ h_{X})/2$, except where otherwise noted. We define for $n \geq 1$, 
\begin{equation}\label{annulus define}
A_{n} = \{x \in B_{\e}(\xi,r) \cap X_{\e}: e^{-\e n}r \leq d_{\e}(x) < e^{-\e(n-1)}r\}.
\end{equation}
and note as before that we have
\[
\mu_{\beta}(B_{\e}(\xi,r)) = \sum_{n=1}^{\infty} \mu_{\beta}(A_{n}). 
\]
The estimates of the proof of Theorem \ref{doubling threshold} then show that we have a constant $c_{0} = c_{0}(\delta,K,\e,M) \geq 0$ such that for each $n \geq 1$ and $x \in A_{n}$ we have $|xz| \leq n+c_{0}$. 

Recall that $R_{0} > 0$ was chosen such that the translates of the ball $B_{0}:=B_{X}(z,R_{0})$ by $\Gamma$ cover $X$. For each $n \geq 1$ we let $\{g_{n,j}\}_{j=1}^{s_{n}} \subset \Gamma$ be a minimal collection of group elements such that the balls $g_{n,j}(B_{0})$ cover $A_{n}$ for $1 \leq j \leq s_{n}$. By minimality we can assume that each of these balls intersects $A_{n}$. Setting $c_{*} = 2R_{0} + c_{0}$, we then have $g_{n,j}(B_{0}) \subset B_{X}(z,n+c_{*})$ for $1 \leq j \leq s_{n}$. In particular $g_{n,j}(p) \in B_{X}(z,n+c_{*})$ for each $n$ and $j$. It follows from \eqref{exponential crit} with $h = (\beta + h_{X})/2$ that 
\[
s_{n} \leq C_{h}e^{h(n+c_{*})} \asymp e^{h n},
\]
since $h > h_{X}$, with the second comparison making the constants implicit. On the other hand, letting $x_{n,j} \in g_{n,j}(B_{0}) \cap A_{n}$ be a point in this intersection, we have $d_{\e}(x_{n,j}) \asymp re^{-\e n}$ and therefore $\rho_{\e}(x_{n,j}) \asymp re^{-\e n}$ by Lemma \ref{compute distance}. Hence $\rho_{\beta}(x_{n,j}) \asymp r^{\beta/\e}e^{-\beta n}$. Since all of the balls $g_{n,j}(B_{0})$ have radius $R_{0}$ and since $\mu$ is $\Gamma$-invariant, the Harnack inequality \eqref{Harnack} implies that 
\[
\mu_{\beta}(g_{n,j}(B_{0})) \asymp r^{\beta/\e}e^{-\beta n} \mu(g_{n,j}(B_{0})) = r^{\beta/\e}e^{-\beta n} \mu(B_{0}).
\] 
Thus we conclude that
\[
\mu_{\beta}(A_{n}) \ls \sum_{j=1}^{s_{n}} \mu_{\beta}(g_{n,j}(B_{0})) \ls r^{\beta/\e}s_{n}e^{-\beta n}\mu(B_{0}) \ls r^{\beta/\e}e^{(h-\beta)n}\mu(B_{0}). 
\]
Since $h < \beta$, we obtain by summing the geometric series that
\[
\mu_{\beta}(B_{\e}(\xi,r)) \ls \sum_{n=1}^{\infty} r^{\beta/\e}e^{(h-\beta)n}\mu(B_{0}) \ls r^{\beta/\e}\mu(B_{0}).
\]
Thus the hypotheses of Proposition \ref{global doubling} hold, so we can conclude the desired doubling estimate for $\mu_{\beta}$. 
\end{proof}


Propsition \ref{crit doubling} proves the doubling claim of Theorem \ref{Riem doubling}, as we will see in the next section. As Remark \ref{renormalize} below indicates, this range for the measure to be doubling is generally sharp.  

\begin{rem}\label{renormalize}
For this remark we assume that we are in the setting of Theorem \ref{Riem doubling}: we let $X$ be a complete simply connected negatively curved Riemannian manifold with sectional curvatures $\leq -1$ and assume that we have a cocompact isometric discrete action by a group $\Gamma$ on $X$. We denote the $\Gamma$-invariant Riemannian volume on $X$ by $\mu$, fix a point $z \in X$, and consider the measure $\mu_{\beta,z}$ on $X$ defined for each $\beta > 0$ by 
\[
d\mu_{\beta,z}(x) = e^{-\beta |xz|}d\mu(x). 
\]
By the theory of Patterson-Sullivan measures (see for instance \cite[Th\'eor\`eme 1.7]{R03}) we have $\mu_{\beta,z}(X) < \infty$ if and only if $\beta > h_{X}$. Since the conformal deformation $X_{1,z}$ of $X$ with conformal factor $\rho_{1,z}(x) = e^{-|xz|}$ is bounded, this implies that $\mu_{\beta,z}$ is not doubling on $X_{1,z}$ when $\beta \leq h_{X}$. If we consider the renormalizations $\bar{\mu}_{\beta,z} = \mu_{\beta,z}(X)^{-1}\mu_{\beta,z}$ of $\mu_{\beta,z}$ for $\beta > h_{X}$ as a measure on $\bar{X}_{1,z} \cong X \cup \p X$ and take the limit as $\beta \rightarrow h_{X}$ then these measures converge in the weak* topology to a measure $\nu_{z}$ on $X \cup \p X$ that is supported on $\p X$; here we are using that the induced topology on $X \cup \p X$ from the identification $\bar{X}_{1,z} \cong X \cup \p X$ coincides with the standard \emph{cone topology} on $X \cup \p X$, see \cite[Remark 4.14(b)]{BHK}. This measure $\nu_{z}$ will be uniformly comparable with the Patterson-Sullivan measure on $\p X$ based at $z$.

\end{rem}

\section{Poincar\'e inequalities for uniformized measures}\label{sec:poincare}

We begin this section by formally introducing Poincar\'e inequalities. We let $(X,d,\mu)$ be a metric measure space with the property that $0 < \mu(B) < \infty$ for all balls $B \subset X$.  For a measurable subset $E \subset X$ satisfying $0 < \mu(E) < \infty$ and a function $u$ that is $\mu$-integrable over $E$ we write 
\begin{equation}\label{average notation}
u_{E} = \dashint_{E} u \, d\mu = \frac{1}{\mu(E)}\int_{E} u \, d\mu
\end{equation}
for the mean value of $u$ over $E$. Let $u: X \rightarrow \R$ be given. A Borel function $g: X \rightarrow [0,\infty]$ is an \emph{upper gradient} for $u$ if for each rectifiable curve $\gamma$ joining two points $x,y \in X$ we have
\[
|u(x)-u(y)| \leq \int_{\gamma} g \, ds.
\]
A measurable function $u: X \rightarrow \R$ is \emph{integrable on balls} if for each ball $B \subset X$ we have that $u$ is integrable over $B$. For a given $p \geq 1$ we say that $X$ \emph{supports a $p$-Poincar\'e inequality} if there are constants $\la \geq 1$ and $C_{\mathrm{PI}} > 0$ such that for each measurable function $u: X \rightarrow \R$ that is integrable on balls, for each ball $B \subset X$, and each upper gradient $g$ of $u$ we have
\begin{equation}\label{uniformly local Poincare}
\dashint_{B} |u-u_{B}|\,d\mu \leq C_{\mathrm{PI}}\diam(B)\left(\dashint_{\la B} g^{p} \, d\mu\right)^{1/p}, 
\end{equation}
for a constant $C_{\mathrm{PI}} > 0$. The constant $\la$ is called the \emph{dilation constant}. If there is a constant $R_{0} > 0$ such that \eqref{uniformly local Poincare} only holds on balls of radius at most $R_{0}$ then we will say that $X$ \emph{supports a $p$-Poincar\'e inequality on balls of radius at most $R_{0}$}. We will also say that $X$ \emph{supports a uniformly local $p$-Poincar\'e inequality}. By H\"older's inequality a metric measure space that supports a $p$-Poincar\'e inequality also supports a $q$-Poincar\'e inequality for each $q \geq p$, and the same is true in regards to supporting a uniformly local $p$-Poincar\'e inequality.

For this section we carry over the same standing hypotheses and notation as discussed at the start of Section \ref{sec:doubling}. We will assume in addition that we are given $p \geq 1$ such that the Gromov hyperbolic space $X$ is equipped with a uniformly locally doubling measure $\mu$ that supports a $p$-Poincar\'e inequality on balls of radius at most $R_{0}$, where $R_{0}$ is the same radius up to which $\mu$ is doubling on $X$. We note that Proposition \ref{enlarge doubling} implies that there is no loss of generality in assuming that these two radii are the same. We will also assume that $\mu_{\beta}$ is doubling on $\bar{X}_{\e}$ for some constant $C_{\mu_{\beta}}$. We will show under these hypotheses that the metric measure space $(\bar{X}_{\e},d_{\e},\mu_{\beta})$ supports a $p$-Poincar\'e inequality with dilation constant $\la = 1$ and constant $C_{\mathrm{PI}}^{*}$ depending only on the uniformization data and the constants $R_{0}$, $C_{\mu}$, $C_{\beta}$, $p$, $\la$, and $C_{\mathrm{PI}}$ associated to the uniformly local doubling property of $\mu$, the global doubling of $\mu_{\beta}$, and the uniformly local $p$-Poincar\'e inequality on $X$. In particular this proves Theorem \ref{global Poincare}. 

The proof splits into two steps. In the first step we show that the $p$-Poincar\'e inequality \eqref{uniformly local Poincare} holds on sufficiently small subWhitney balls in the metric measure space $(X_{\e},d_{\e},\mu_{\beta})$. The proof is essentially identical to \cite[Lemma 6.1]{BBS20}. In the statement and proof of Lemma \ref{Whitney Poincare} ``the data" refers to the uniformization data and the constants $R_{0}$, $C_{\mu}$,  $p$, $\la$, and $C_{\mathrm{PI}}$. For Lemma \ref{uniformly local Poincare} we do not need to assume that $\mu_{\beta}$ is doubling. We will require the following easy lemma. 



\begin{lem}\label{switching lemma}\cite[Lemma 4.17]{BB11}
Let $u: X \rightarrow \R$ be integrable, let $p \geq 1$, let $\alpha \in \R$, and let $E \subset X$ be a measurable set with $0 < \mu(E) < \infty$. Then
\[
\left(\dashint_{E} |u-u_{E}|^{p}\,d\mu\right)^{1/p} \leq 2\left(\dashint_{E} |u-\alpha|^{p}\,d\mu\right)^{1/p}
\]
\end{lem}

\begin{lem}\label{Whitney Poincare}
There exists $c_{0} > 0$ depending only on the uniformization data and $R_{0}$ such that for all $x \in X_{\e}$ and all $0 < r\leq c_{0}d_{\e}(x)$ the $p$-Poincar\'e inequality \eqref{uniformly local Poincare} for $\mu_{\beta}$ holds on the ball $B_{\e}(x,r)$ with dilation constant $\hat{\la}$ and constant $\hat{C}_{\mathrm{PI}}$ depending only on the data.
\end{lem}

\begin{proof}
Put $B_{\e} = B_{\e}(x,r)$ with $0 < r \leq c_{0}d_{\e}(x)$, where $0 < c_{0} \leq \frac{1}{2}$ is a constant to be determined. Let $C_{*}$ be the constant of Lemma \ref{sub inclusion}. We choose $c_{0} > 0$ small enough that $c_{0}C_{*}^{2} \leq \frac{1}{2}$. We conclude by applying Lemma \ref{sub inclusion} twice that 
\begin{equation}\label{ball chain}
B_{\e} \subset B:=B_{X}\left(x,\frac{C_{*}r}{\rho_{\e}(x)}\right)  \subset B_{\e}\left(x,C_{*}^{2}r\right) = \hat{\la} B_{\e},
\end{equation}
with $\hat{\la} = C_{*}^{2}$, since 
\[
C_{*}^{2}r \leq c_{0}C_{*}^{2}d_{\e}(x) \leq \frac{1}{2}d_{\e}(x).
\]
Moreover by \eqref{comparison chain} we see that for all $y \in \hat{\la} B_{\e}$ we have $\rho_{\beta}(y) \asymp \rho_{\beta}(x)$ with comparison constant depending only on the uniformization data.

Now let $u$ be a function on $X_{\e}$ that is integrable on balls and let $g_{\e}$ be an upper gradient of $u$ on $X_{\e}$. By the same basic calculation as in \cite[(6.3)]{BBS20} we have that $g:=g_{\e}\rho_{\e}$ is an upper gradient of $u$ on $X$. For $c_{0}$ sufficiently small (depending only on the uniformization data and $R_{0}$) we will have by Lemma \ref{compute distance} that
\[
\frac{C_{*}r}{\rho_{\e}(x)} \leq \frac{C_{*}c_{0}d_{\e}(x)}{\rho_{\e}(x)} \leq R_{0}.
\]
Thus the $p$-Poincar\'e inequality \eqref{uniformly local Poincare} (for $\mu$) holds on $B$. Since $\rho_{\beta}(y) \asymp \rho_{\beta}(x)$ on $\hat{\la} B_{\e}$ with comparison constant depending only on the uniformization data (by \eqref{comparison chain}) we have that
\begin{equation}\label{measure comparison}
\mu_{\beta}(B) \asymp \rho_{\beta}(x)\mu(B),
\end{equation}
with comparison constant depending only on the uniformization data, and the same comparison holds with either $B_{\e}$ or $\hat{\la}B_{\e}$ replacing $B$. Writing $u_{B,\mu} = \dashint_{B} u \, d\mu$, we conclude by using the inclusions of \eqref{ball chain}, the measure comparison \eqref{measure comparison}, and the $p$-Poincar\'e inequality for $\mu$ on $B$,
\begin{align*}
\dashint_{B_{\e}}|u-u_{B,\mu}| \, d\mu_{\beta} &\lesssim \dashint_{B} |u-u_{B,\mu}| d\mu \\
&\leq \frac{2C_{\mathrm{PI}}C_{*}r }{\rho_{\e}(x)}\left(\dashint_{B} g^{p}\, d\mu\right)^{1/p} \\
&\asymp \frac{r}{\rho_{\e}(x)}\left(\dashint_{B} (g_{\e}\rho_{\e})^{p}\, d\mu_{\beta}\right)^{1/p} \\
&\lesssim r \left(\dashint_{\hat{\la} B_{\e}} g_{\e}^{p} \, d\mu_{\beta}\right)^{1/p},
\end{align*}
where all implied constants depend only on the data. By Lemma \ref{switching lemma} we can replace $u_{B,\mu}$ with $u_{B_{\e},\mu_{\beta}} = \dashint_{B_{\e}} u \, d\mu_{\beta}$ on the left to conclude the proof of the lemma. 
\end{proof}

The second part of the proof is the following key proposition.

\begin{prop}\label{uniform upgrade}\cite[Proposition 6.3]{BBS20} 
Let $\Omega$ be an $A$-uniform metric space equipped with a doubling measure $\nu$ such that there is a constant $0 < c_{0} < 1$ for which the $p$-Poincar\'e inequality \eqref{uniformly local Poincare} holds for fixed constants $C_{\mathrm{PI}}$ and $\la$ on all subWhitney balls $B$ of the form $B = B_{\Omega}(x,r)$ with $x \in \Omega$ and $0 < r \leq c_{0}d_{\Omega}(x)$. Then the metric measure space $(\Omega,d,\nu)$ supports a $p$-Poincar\'e inequality with dilation constant $A$ and constant $C'_{\mathrm{PI}}$ depending only on $A$, $c_{0}$, $p$, $C_{\mathrm{PI}}$, $\la$, and the doubling constant $C_{\nu}$ for $\nu$. 
\end{prop}

This proposition is stated for bounded $A$-uniform metric spaces in \cite{BBS20} but the proof works without modification for unbounded $A$-uniform metric spaces provided that the doubling property of $\nu$ holds at all scales and the $p$-Poincar\'e inequality on subWhitney balls hold at all appropriate scales. 

We can now verify the global $p$-Poincar\'e inequality on $\bar{X}_{\e}$, which proves Theorem \ref{global Poincare}. Below ``the data" includes all the constants from Lemma \ref{Whitney Poincare} as well as the doubling constant $C_{\mu_{\beta}}$ for $\mu_{\beta}$. 

\begin{proof}[Proof of Theorem \ref{global Poincare}]
By Lemma \ref{Whitney Poincare} there is a $c_{0} > 0$ determined only by the data such that the $p$-Poincar\'e inequality holds on subWhitney balls of the form $B_{\e}(x,r)$ with $0 < r \leq c_{0}d_{\e}(x)$ for $x \in X$, with uniform constants $\hat{C}_{\mathrm{PI}}$ and $\hat{\la}$. Since $(X_{\e},d_{\e})$ is an $A$-uniform metric space with $A = A(\delta,K,\e,M)$  and we assumed $\mu_{\beta}$ is globally doubling on $X_{\e}$ with constant $\mu_{\beta}$, it follows from Proposition \ref{uniform upgrade} that the metric measure space $(X_{\e},d_{\e},\mu_{\beta})$ supports a $p$-Poincar\'e inequality with constant $C_{\mathrm{PI}}'$ depending only on the data and dilation constant $A$. Since $X_{\e}$ is geodesic it follows that the $p$-Poincar\'e inequality \eqref{uniformly local Poincare} in fact holds with dilation constant $1$, with constant  $C_{\mathrm{PI}}^{*}$ depending only on the data \cite[Theorem 4.18]{Hein01}.

By \cite[Lemma 8.2.3]{HKST} we conclude that the completion $(\bar{X}_{\e},d_{\e},\mu_{\beta})$ (with $\mu_{\beta}(\p X_{\e}) = 0$) also supports a $p$-Poincar\'e inequality with constants depending only on the constants for the $p$-Poincar\'e inequality on $X_{\e}$ and the doubling constant of $\mu_{\beta}$. Since $\bar{X}_{\e}$ is also geodesic it follows by the same reasoning \cite[Theorem 4.18]{Hein01} that we can take the dilation constant to be $1$ in this case as well. 
\end{proof}

We can now prove Theorem \ref{Riem doubling} as well. 

\begin{proof}[Proof of Theorem \ref{Riem doubling}]  
Let $X$ be a complete simply connected $n$-dimensional Riemannian manifold $X$ with sectional curvatures $\leq -1$ that is equipped with a cocompact discrete isometric action of a group $\Gamma$. Then $X$ is $\delta$-hyperbolic with $\delta = \delta(\mathbb{H}^{2})$ being the same as that of the hyperbolic plane $\mathbb{H}^{2}$ of constant negative curvature $-1$ \cite[p. 169]{BH99}. Let $\mu$ be the $\Gamma$-invariant Riemannian volume on $X$.  The space $X$ is $0$-roughly starlike from any point of $X \cup \p X$ since any geodesic $\gamma: I \rightarrow X$ defined on any interval $I \subset \R$ can be uniquely extended to a full geodesic line $\gamma: \R \rightarrow X$. By \cite[Theorem 1.10]{Bu20} the densities $\rho_{1,b}$ for $b \in \hat{\mathcal{B}}(X)$ are GH-densities with a uniform constant $M$. Thus we can apply the results of the previous sections here with this constant $M$ and $\delta = \delta(\mathbb{H}^{2})$, $K = 0$, and $\e = 1$. 

Choose $R_{0} > 0$ large enough that for each $x \in X$ the translates of the ball $B_{X}(x,R_{0})$ by $\Gamma$ cover $X$. On each such ball $B_{X}(x,R_{0})$ the Riemannian metric on $X$ is biLipschitz to the standard Euclidean metric on the unit ball in $\R^{n}$ with biLipschitz constant independent of $x$ (by the cocompactness of $\Gamma$) and the Riemannian volume is uniformly comparable to the standard $n$-dimensional Lebesgue measure. Since $\R^{n}$ equipped with the $n$-dimensional Lebesgue measure is a doubling metric measure space that supports a $1$-Poincar\'e inequality \cite[Chapter 4]{Hein01}, it follows that $X$ equipped with $\mu$ is uniformly locally doubling and supports a uniformly local $1$-Poincar\'e inequality. We remark that all of the parameters considered so far are independent of the choice of $b \in \hat{\mathcal{B}}(X)$. 

We conclude by Proposition \ref{crit doubling} that for each $\beta > h_{X}$ the metric measure space $(\bar{X}_{1,b},d_{1,b},\mu_{\beta,b})$ is doubling with a uniform doubling constant $C_{\mu_{\beta}}$ independent of the choice of $b \in \hat{\mathcal{B}}(X)$. The $1$-Poincar\'e inequality on $(X_{1,b},d_{1,b},\mu_{\beta,b})$ and $(\bar{X}_{1,b},d_{1,b},\mu_{\beta,b})$ then follows from Theorem \ref{global Poincare}.
\end{proof}

\section{Uniform inversion}\label{sec:inversion}

In this section we consider a procedure that we will call \emph{uniform inversion} that can be used to convert bounded uniform metric spaces into unbounded uniform metric spaces and vice versa. This procedure can be thought of as a variation of the inversion procedure considered in \cite{HSX08} that is specialized to the context of uniform metric spaces. We show that this procedure can be extended to measures in such a way that it preserves the doubling property and $p$-Poincar\'e inequalities for a given $p \geq 1$. For general metric measure spaces it was shown by Li and Shanmugalingam \cite{LS15} that the doubling property can be preserved under sphericalization and inversion, however they were only able to obtain preservation of $p$-Poincar\'e inequalities under the additional assumption that the space was annularly quasiconvex. This condition excludes many uniform metric spaces such as those that are obtained by uniformizing trees. With a weaker assumption Durand-Cartagena and Li \cite{DL15} showed that $p$-Poincar\'e inequalities can be preserved once $p$ is sufficiently large. Using the results of the previous sections we will show that uniform inversion preserves $p$-Poincar\'e inequalities for \emph{all} $p \geq 1$. 

Let $(\Omega,d)$ be an $A$-uniform metric space, $A \geq 1$. We denote the distance to the metric boundary of $\Omega$ by $d(x):=d_{\Omega}(x)$ for $x \in \Omega$. The \emph{quasihyperbolic metric} on $\Omega$ is defined by, for $x,y \in \Omega$, 
\begin{equation}\label{quasihyperbolic metric}
k(x,y) = \inf \int_{\gamma} \frac{ds}{d(\gamma(s))},
\end{equation}
where the infimum is taken over all rectifiable curves joining $x$ to $y$. The metric space $Y = (\Omega,k)$ is called the \emph{quasihyperbolization} of the metric space $(\Omega,d)$. We note that $Y$ can equivalently be thought of as the conformal deformation of $\Omega$ with conformal factor $\rho(x) = d(x)^{-1}$. The quasihyperbolication $Y$ is a proper geodesic $\delta$-hyperbolic space by \cite[Theorem 3.6]{BHK} with $\delta = \delta(A)$ depending only on $A$.

To precisely state our claims below we introduce the following ratio when $\Omega$ is bounded,
\begin{equation}\label{tightness}
\phi(\Omega):= \frac{\diam \, \Omega}{\diam \, \p \Omega},
\end{equation}
where we define $\phi(\Omega) = \infty$ if $\p \Omega$ contains only one point. Until the end of this section we will always assume that $\phi(\Omega) < \infty$ if $\Omega$ is bounded, i.e., that $\p \Omega$ contains at least two points. 

\begin{rem}\label{one point}
The case of bounded $\Omega$ with $\p \Omega$ containing only one point is rather degenerate so we will not discuss it here. For instance if $\Omega = [0,1)$ then its quasihyperbolization $Y$ is isometric to $[0,\infty)$, and a Busemann function $b$ on $[0,\infty)$ based at the only point $\infty$ in the Gromov boundary of $[0,\infty)$ is given by $b(t) = -t$ for $t \in [0,\infty)$. By direct calculation we then see for every $\e > 0$ that $Y_{\e,b}$ is also isometric to $[0,\infty)$. In particular $Y_{\e,b}$ is actually a complete metric space, so it can't be a uniform metric space.
\end{rem}

When $\Omega$ is unbounded there is a constant $K = K(A)$ depending only on $A$ such that $Y$ is $K$-roughly starlike from any point of $Y \cup \p Y$, while when $\Omega$ is bounded there is a constant $K = K(A)$ such that $Y$ is $K$-roughly starlike from any $z \in \Omega$ such that $d(z) = \sup_{x \in \Omega}d(x)$, and a constant $K' = K'(A,\phi(\Omega))$ depending only on $A$ and the ratio $\phi(\Omega)$ such that $Y$ is $K'$-roughly starlike from any point of $Y \cup \p Y$ \cite[Proposition 3.3]{Bu21}. The dependence of $K'$ on $\phi(\Omega)$ in the bounded case is necessary by \cite[Example 3.4]{Bu21}. 

From the discussion after inequality \eqref{GH inequality} we can find an $\e = \e(A) > 0$ depending only on $A$ (since $Y$ is $\delta$-hyperbolic with $\delta = \delta(A)$) such that for any $b \in \hat{\mathcal{B}}(Y)$ the density $\rho_{\e,b}(x) = e^{-\e b(x)}$ on $Y$ is a GH-density with constant $M = 20$. We will fix such an $\e$ for each value of $A$ for the rest of this section.  


\begin{defn}\label{def:uniform inversion}
For a given $b \in \hat{\mathcal{B}}(Y)$ we let $\Omega_{b}= Y_{\e,b}$ denote the conformal deformation of $Y$ with conformal factor $\rho_{\e,b}$. We will refer to the metric space $\Omega_{b}$ as the \emph{uniform inversion} of $\Omega$ based at $b$. 
\end{defn}

The next proposition shows that uniform inversions of $\Omega$ have the properties suggested by their name. 

\begin{prop}\label{inversion properties}
Let $\Omega$ be an $A$-uniform metric space. Let $Y = (\Omega,k)$ be the quasihyperbolization of $\Omega$. Then $\Omega_{b}$ is an $A'$-uniform metric space for each $b \in \hat{\mathcal{B}}(Y)$ with $A'= A'(A)$ if $\Omega$ is unbounded and $A' = A'(A,\phi(\Omega))$ if $A'$ is unbounded. Furthermore $\Omega_{b}$ is bounded if and only if $b \in \mathcal{D}(Y)$. 
\end{prop} 

All of the claims of Proposition \ref{inversion properties} follow from applying \cite[Theorem 1.1]{Bu20} to $Y$, since $Y$ is $\delta = \delta(A)$-hyperbolic, $K = K(A)$-roughly starlike from any point of $Y \cup \p Y$ if $\Omega$ is unbounded (with $K = K(A,\phi(\Omega))$ instead if $\Omega$ is bounded) and $\rho_{\e,b}$ is a GH-density with constant $M = 20$ (and $\e = \e(A)$). In particular uniform inversion can be used to produce an unbounded uniform metric space $\Omega_{b}$ from a bounded uniform metric space $\Omega$ by choosing $b \in \mathcal{B}(Y)$, and similarly can be used to produce a bounded uniform metric space $\Omega_{b}$ from an unbounded uniform metric space $\Omega$ by choosing $b \in \mathcal{D}(Y)$. 

The primary reason to consider uniform sphericalization and inversion is that these operations can be extended to measures in such a way as to preserve the doubling property and the $p$-Poincar\'e inequality for all $p \geq 1$. This comes at a price of increased complexity of these operations as opposed to the standard inversion operation, with the loss of several nice features of the latter that are obtained in \cite{BHX08}. We remark that one can show that the identity map $\Omega \rightarrow \Omega_{b}$ is always quasim\"obius for any $b \in \hat{\mathcal{B}}(Y)$, as is true of ordinary inversion; see \cite[Proposition 4.4]{Bu21}.

Now suppose in addition that $\Omega$ is equipped with a Borel measure $\nu$ that is doubling and satisfies $0 < \nu(B) < \infty$ for all balls $B \subset \Omega$. We write $C_{\nu}$ for the doubling constant of $\nu$. For each $\alpha > 0$ we define a measure $\mu^{\alpha}$ on $\Omega$ by 
\[
d\mu^{\alpha}(x) = d(x)^{-\alpha}d\nu(x),
\]
and consider $\mu$ as a measure on $Y$. Then \cite[Proposition 7.3]{BBS20} shows for each $\alpha > 0$ that $\mu$ is doubling on balls of radius at most $R_{0} = 1$ with local doubling constant $C_{\mu^{\alpha}}$ depending only on $A$ and $\alpha$. We let $\beta_{0} > 0$ be the exponent determined by applying Theorem \ref{doubling threshold} to the quasihyperbolization $Y$ equipped with the measure $\mu^{\alpha}$ in relation to its uniformization $\Omega_{b} = Y_{\e,b}$ with $b \in \hat{\mathcal{B}}(Y)$ and $\e = \e(A)$. We then choose $\beta \geq \beta_{0}$ and set $\nu_{\alpha,\beta,b} = (\mu^{\alpha})_{\beta,b}$ to be the measure obtained from $\mu^{\alpha}$ by applying the formula \eqref{define beta} with our chosen $b \in \hat{\mathcal{B}}(Y)$. We consider $\nu_{\alpha,\beta,b}$ as defining a two parameter family of measures on $\Omega_{b}$ and write $d_{b}$ for the metric on $\Omega_{b}$. Applying Theorems \ref{doubling threshold} and \ref{global Poincare} to this family yields the following theorem. 

\begin{thm}\label{inversion doubling}
Let $(\Omega,d,\nu)$ be a doubling metric measure space with doubling constant $C_{\nu}$ such that $(\Omega,d)$ is an $A$-uniform metric space with $\phi(\Omega) < \infty$ if $\Omega$ is bounded. Let $Y = (\Omega,k)$ be the quasihyperbolization of $\Omega$. Then for each $\alpha > 0$ there is a constant $\beta_{0} = \beta_{0}(\alpha,A,C_{\nu})$ (if $\Omega$ is unbounded) or $\beta_{0} = \beta_{0}(\alpha,A,C_{\nu},\phi(\Omega))$ (if $\Omega$ is bounded) such that for any $b \in \hat{\mathcal{B}}(Y)$ and any $\beta \geq \beta_{0}$ we have that $\nu_{\alpha,\beta,b}$ is doubling on $\Omega_{b}$ with doubling constant $C_{\nu_{\alpha,\beta,b}}$ depending only on $\alpha$, $\beta$, $A$, and $C_{\nu}$ (and $\phi(\Omega)$ if $\Omega$ is bounded). 

If furthermore the metric measure space $(\Omega,d,\nu)$ supports a $p$-Poincar\'e inequality for a given $p \geq 1$ then the metric measure space $(\Omega_{b},d_{b},\nu_{\alpha,\beta,b})$ supports a $p$-Poincar\'e inequality with constants depending only on $\alpha$, $\beta$, $A$, $C_{\nu}$, $p$, and the constants in \eqref{uniformly local Poincare} (and $\phi(\Omega)$ if $\Omega$ is bounded). 
\end{thm}

The final claim follows from the fact that under the hypotheses of the proposition the metric measure space $(Y,k,\mu^{\alpha})$ supports a uniformly local $p$-Poincar\'e inequality by \cite[Proposition 7.4]{BBS20} with radius and constants depending only on $\alpha$, $A$, $C_{\nu}$, $p$, and the constants in the $p$-Poincar\'e inequality for $(\Omega,d,\nu)$. Hence we can directly apply Theorem \ref{global Poincare} to $(Y_{\e,b},k_{\e,b},(\mu^{\alpha})_{\beta,b}) = (\Omega_{b},d_{b},\nu_{\alpha,\beta,b})$ in this case. 

We remark that it is not immediately clear what choices of $\alpha$ and $\beta$ are natural in the context of Proposition \ref{inversion doubling}, which is why we have left them as free parameters. Theorem \ref{inversion doubling} shows that $(\Omega_{b},d_{b},\nu_{b,\beta,\alpha})$ is always a doubling metric measure space once $\beta$ is large enough in relation to $\alpha$, and that $p$-Poincar\'e inequalities transfer over to this space from $(\Omega,d,\nu)$.

\bibliographystyle{plain}
\bibliography{GlobalDoubling}

\end{document}